\documentclass[10pt,reqno]{amsart}
\usepackage[british]{babel}
\usepackage{amsmath,amsfonts,amssymb,amstext,amsbsy,amsopn,amsthm,amsxtra,amscd} 
\usepackage{graphicx}
\usepackage{pdfpages}
\usepackage{caption}
\usepackage{subcaption}
\usepackage{color}
\usepackage{hyperref}
\usepackage{xcolor}
\usepackage{enumerate}

\numberwithin{equation}{section}

\newtheorem{theorem}{Theorem}[section]
\newtheorem{proposition}[theorem]{Proposition}
\newtheorem{corollary}[theorem]{Corollary}
\newtheorem{lemma}[theorem]{Lemma}
\newtheorem{definition}[theorem]{Definition}
\newtheorem{remark}[theorem]{Remark}
\newtheorem{claim}[theorem]{Claim}

\newtheorem{notation}[theorem]{Notation}
\newtheorem{conjecture}[theorem]{Conjecture}

\newcommand{\bb}[1]{\mathbb{#1}}

\DeclareMathOperator{\supp}{Supp}
\DeclareMathOperator{\conv}{Conv}
\DeclareMathOperator{\ann}{Ann}

\DeclareMathOperator{\dist}{Dist}
\DeclareMathOperator{\nexpl}{NEL}
\DeclareMathOperator{\nexpd}{ONED}
\DeclareMathOperator{\pat}{Pat}

\newcommand{\sob}[1]{\ensuremath{{\mathbin |}\raise-.4ex\hbox{$#1$}}}

\title[ON PER. DECOMP., ONE-SIDED NONEX. DIR. AND NIVAT'S CONJECTURE]{ON PERIODIC DECOMPOSITIONS, ONE-SIDED NONEXPANSIVE DIRECTIONS AND NIVAT'S CONJECTURE}
\author{Cleber Fernando Colle}
\address{Rio de Janeiro State University, Rio de Janeiro, Brazil}
\email{cleber.colle@ime.uerj.br}


\begin{document}
\begin{abstract}
Nivat's conjecture is a famous open problem in symbolic dynamics. The existence of nonexpansive lines that when endowed with a given orientation are one-sided nonexpansive directions is at the heart of some advances. In his Ph.D. thesis, Michal Szabados conjectured that for a minimal periodic decomposition the nonexpansive lines are exactly the lines that contain a period of some periodic configuration in such decomposition. In this paper, we provide conditions where (i) Szabados's conjecture holds and (ii) a given line is nonexpansive if and only if the same line endowed with a given orientation is a one-sided nonexpansive direction. As a corollary of our main result, we get that Nivat's conjecture holds for low convex complexity configurations if and only if it holds for low convex complexity configurations satisfying (i) and (ii).
	
\end{abstract}

\maketitle 

\section{Introduction and preliminaries}

\subsection{Periodicity and complexity} 

Let $\mathcal{A}$ be an alphabet with at least two elements. The elements of $\mathcal{A}^{\bb{Z}^d}$, called \emph{configurations}, have the form $\eta = (\eta_{g})_{g \in \bb{Z}^d}$, where $\eta_g \in \mathcal{A}$ for all $g \in \bb{Z}^d$. For each \(u \in \bb{Z}^d\), we define the \emph{shift application} by \((T^u\eta)_g = \eta_{g+u}\) for all \(g \in \bb{Z}^d\). 

A configuration $\eta \in \mathcal{A}^{\bb{Z}^d}$ is said to be \emph{periodic} if there exists a vector $h \in (\bb{Z}^d)^*$, called \emph{period of $\eta$}, such that $\eta_{g+h} = \eta_g$ for all $g \in \bb{Z}^d$. If \(\eta\) has \(d\) periods linearly independents over \(\bb{R}^d\), we say that \(\eta\) is \emph{fully periodic}.

From now on, we will suppose that $\mathcal{A}$ is a finite alphabet. Let $\mathcal{S} \subset \bb{Z}^d$ be a non-empty, finite set. For a configuration $\eta \in \mathcal{A}^{\bb{Z}^d}$, the \emph{$\mathcal{S}$-com\-plexity of $\eta$}, denoted by $P_{\eta}(\mathcal{S})$, is defined to be the number of distinct elements in $\mathcal{A}^{\mathcal{S}}$ of the form $(T^u\eta)\sob{\mathcal{S}}$, called \emph{\(\mathcal{S}\)-patterns of \(\eta\)}, where $\cdot \sob{\mathcal{S}}$ means the restriction to the set $\mathcal{S}$, i.e., for a con\-fig\-u\-ra\-tion $\vartheta \in \mathcal{A}^{\bb{Z}^d}$, \(\vartheta\sob{\mathcal{S}} := (\vartheta_g)_{g \in \mathcal{S}}\). If we define \[\pat(\mathcal{S},\eta) := \left\{(T^u\eta)\sob{\mathcal{S}} : u \in \bb{Z}^d\right\},\] then $P_{\eta}(\mathcal{S}) = |\pat(\mathcal{S},\eta)|$. We say that $\eta$ has \emph{low complexity} if $P_{\eta}(\mathcal{S}) \leq |\mathcal{S}|$ holds for some non-empty, finite set $\mathcal{S} \subset \bb{Z}^d$. If we have in addition that \(\mathcal{S}\) is a convex set, where by \emph{convex} we mean a subset of $\bb{Z}^d$ whose convex hull in $\bb{R}^d$, denoted by $\conv(\mathcal{S})$, is closed and $\mathcal{S} = \conv(\mathcal{S}) \cap \bb{Z}^d$, we say that $\eta$ has \emph{low convex complexity}. For $$I_{n_1} \times \cdots \times I_{n_d},$$ where $I_n := \{0,1,\ldots,n-1\}$ and $n \in \bb{N}$\footnote{We use $\bb{N} = \{1,2,\ldots\}$ and $\bb{Z}_+ = \bb{N} \cup \{0\}$.}, we write for short $P_{\eta}(n_1, \ldots ,n_d)$ instead of $P_{\eta}(I_{n_1} \times \cdots \times I_{n_d})$.

In the one-dimensional case, $I_n$-pat\-terns are called words and configurations are classically called sequences. Morse and Hedlund~\cite{hedlund} proved in 1938 one of the most famous results in symbolic dynamics which establishes a connection between periodic sequences and complexity:

\begin{theorem}[Morse-Hedlund]
Given a sequence $\xi \in \mathcal{A}^{\bb{Z}}$, if there exists $n \in \bb{N}$ such that $P_{\xi}(n) \leq n$, then \(\xi\) is periodic.
\end{theorem} 

If we naturally extend the notions of periodicity and complexity for sequences $\xi \in \mathcal{A}^{B}$, where $B = \{a,a+1,\ldots\}$ and \(a \in \bb{Z}\) (see \cite{van} for details), a version of Morse-Hedlund's Theorem also holds:

\begin{theorem}[Morse-Hedlund]\label{Morse-HedlundThm2}
Given a sequence $\xi \in \mathcal{A}^{B}$, with $B = \{a,a+1,\ldots\}$ and \(a \in \bb{Z}\), if there exists \(n \in \bb{N}\) such that \(P_{\xi}(n) \leq n\), then the sequence \((\xi_i)_{i \in B+n}\) is periodic.
\end{theorem}

Proposed by Maurice Nivat at ICALP 1997 in Bologna, the so-called Nivat's Con\-jecture \cite{nivat} is a generalization of the Morse-Hedlund Theorem for the two-di\-men\-si\-onal case:

\begin{conjecture}[Nivat]
Given a configuration $\eta \in \mathcal{A}^{\bb{Z}^2}$, if there exist $n,k \in$ $\bb{N}$ such that $P_{\eta}(n,k)  \leq nk$, then $\eta$ is periodic.
\end{conjecture}

Julien Cassaigne~\cite{cassaigne1} conjectured that the analogous of Nivat's conjecture holds for convex subsets of $\bb{Z}^2$.

The Nivat's conjecture has been extensively studied in the last 16 years. The first step towards a proof for the conjecture was given by Sander and Tijdeman~\cite{tijdeman}. They showed that if $P_{\eta}(n,2) \leq 2n$ for some $n \in \bb{N}$, then $\eta \in \mathcal{A}^{\bb{Z}^2}$ is periodic. Sander and Tijdeman \cite{sander} also found counter-examples to the analogue of Nivat's conjecture in higher dimensions, i.e., they showed that for $d\geq 3$ there exist non-periodic configurations $\eta \in$ $\{0,1\}^{\bb{Z}^d}$ such that $P_{\eta}(n, \ldots, n) = 2n^{d-1}+1$. Julien Cassaigne~\cite{julien06} showed that the Nivat's conjecture does not hold for $d \geq 3$ even if we suppose in addition that the configuration is repetitive, i.e., when the closure of its $\bb{Z}^d$-orbit is a minimal subshift. 

Let $\eta \in \mathcal{A}^{\bb{Z}^2}$ and suppose there exist $n,k \in \bb{N}$ such that $P_{\eta}(n,k) \leq \frac{1}{C}nk$. It was proved that $\eta$ is periodic for $C = 144$ in \cite{epifanio} and for $C = 16$ in \cite{quas}. Using the notion of expansive subspaces introduced by Boyle and Lind~\cite{boyle}, Bryna Kra and Van Cyr~\cite{van} shed a new light towards a proof for Nivat's conjecture by relating expansive subspaces to periodicity. They proved that $\eta$ is periodic for $C = 2$. Following the Cyr and Kra approach's, Colle and Garibaldi~\cite{colle} showed that, for a configuration $\eta \in \mathcal{A}^{\bb{Z}^2}$ that contains all letters of $\mathcal{A}$, if $P_{\eta}(\mathcal{S}) \leq \frac{1}{2}|\mathcal{S}|+|\mathcal{A}|-1$ for some non-empty, finite set $\mathcal{S} \subset$ $\mathbb{Z}^2$ whose convex hull on $\mathbb{R}^2$ is described by pairs of edges with identical size, then $\eta$ is periodic.

Employing results from algebraic geometry, Jarkko Kari and Michal Szabados~\cite{KariSzabados,kari} proved that any low complexity configuration $\eta \in \mathcal{A}^{\bb{Z}^d}$, with $\mathcal{A} \subset \bb{Z}$, has a rigid structure (see Theorem \ref{theorKS}). In the two-dimensional case, they showed that, if $P_{\eta}(n,k) \leq nk$ holds for infinitely many pairs $n,k \in \bb{N}$, then $\eta$ is periodic. One of the most important recent results related to Nivat's conjecture is due to Jarkko Kari and Etienne Moutot \cite{Moutot19}. They proved that any low convex complexity configuration has a periodic accumulation point, which solves Nivat's conjecture in the repetitive case.

Further partial results related to Nivat's conjecture are given in \cite{van1,durand,Szabados,Moutot,Moutot21}. 

\subsection{Expansiveness and one-sided expansiveness} 

Let $F$ be a subset of $\bb{R}^d$. For $g \in \bb{Z}^d$\!, we consider \[\dist(g,F) = \inf \{\|g-u\| : u \in F\},\] where $\| \cdot \|$ is the Euclidean norm in $\bb{R}^d$. Given $t > 0$, the $t$-neighbourhood of $F$ is defined as \[F^{t} := \{g \in \bb{Z}^d : \dist(g,F) \leq t\}.\] Let $X \subset \mathcal{A}^{\bb{Z}^d}$ be a subshift, i.e., a closed subset invariant by the $\bb{Z}^d$-action $T^u$, $u \in \bb{Z}^d$. Following Boyle and Lind \cite{boyle}, we say that a subspace $F \subset \bb{R}^d$ is \emph{expansive} on $X$ if there exists $t>0$ such that \[\forall \ x,y \in X, \quad x\sob{F^{t}} = y\sob{F^{t}} \implies x = y.\] Other\-wise, the subspace $F$ is called \emph{nonexpansive} on $X$. Boyle and Lind~\cite[Theorem~3.7]{boyle} showed that if $X \subset \mathcal{A}^{\bb{Z}^d}$ is an infinite sub\-shift, then, for each $0 \leq n < d$, there exists an $n$-dimensional subspace of $\bb{R}^d$ that is nonexpansive on $X$.

From now on, we will focus on the two-dimensional case and, for a given configuration $\eta \in \mathcal{A}^{\bb{Z}^2}$, we will consider the subshift $X_{\eta} :=  \overline{Orb\,(\eta)}$, where $Orb\,(\eta) := \{T^{u}\eta : u \in \bb{Z}^2\}$ de\-no\-tes the $\bb{Z}^2$-orbit of $\eta$ and the bar denotes the closure with respect to the product topology.

\begin{notation}
Given a configuration $\eta \in \mathcal{A}^{\bb{Z}^2}$, we use $\nexpl(\eta)$ to denote the set for\-med by the lines through the origin (one-dimensional subspaces) that are nonexpansive on $X_{\eta}$.
\end{notation}

When $\eta \in \mathcal{A}^{\bb{Z}^2}$ is a periodic configuration, but not fully periodic, the unique line $\ell \in \nexpl(\eta)$, which exists due to Boyle-Lind Theorem, must necessarily contain every period for $\eta$. Hence, any configuration $\eta \in \mathcal{A}^{\bb{Z}^2}$ where $\nexpl(\eta)$ has at least two elements can not be periodic.

In the sequel, we will revisit a more refined version of expansiveness called one-sided nonexpansiveness and introduced by Cyr and Kra in \cite{van}.

We recall that two vectors are parallel if they have the same direction and antiparallel if they have opposite directions. Two oriented objects in $\bb{R}^2$ are said to be \emph{(anti)parallel} if the adjacent vectors to their respective orientations are (anti)par- allel, where by object we mean an oriented line, an oriented line segment or a vector. In a slight abuse of notation, we view an oriented line also as a subset of $\bb{R}^2$.
 
Let $\pmb{\ell} \subset \bb{R}^2$ be an oriented line. The half plane determined by $\pmb{\ell}$ is defined as \[\mathcal{H}(\pmb{\ell}) := \{g \in \bb{Z}^2 : \langle g,(-u_2,u_1) \rangle \geq 0\},\] where $(u_1,u_2) \in \bb{R}^2$ is a non-zero vector parallel to $\pmb{\ell}$ and $\langle \cdot, \cdot \rangle$ denotes the standard inner product of $\bb{R}^2$. We remark that, following the orientation of $\pmb{\ell}$, the interior of $\mathcal{H}(\pmb{\ell})$ is on the left.

\begin{definition}
Let $X \subset \mathcal{A}^{\bb{Z}^2}$ be a subshift. An oriented line $\pmb{\ell} \subset \bb{R}^2$ through the origin is called a one-sided expansive direction on $X$ if \[\forall \ x,y \in X, \quad x\sob{\mathcal{H}(\pmb{\ell})} = y\sob{\mathcal{H}(\pmb{\ell})} \ \Longrightarrow \ x = y.\] Otherwise, the oriented line $\pmb{\ell}$ is called a one-sided non\-ex\-pan\-si\-ve direction on $X$. 
\end{definition}

\begin{notation}
Given a configuration $\eta \in \mathcal{A}^{\bb{Z}^2}$, we use $\nexpd(\eta)$ to denote the set formed by the oriented lines through the origin that are one-sided non\-ex\-pan\-si\-ve di\-rec\-tions on $X_{\eta}$.
\end{notation}

\begin{notation}
\medspace 	
\begin{enumerate}[(i)]\setlength{\itemsep}{5pt}
	\item Given a line $\ell \subset \bb{R}^2$, we use $\pmb{\ell}$ to denote the line $\ell$ endowed of a given o\-ri\-en\-ta\-tion.
	\item Given an oriented line $\pmb{\ell} \subset \bb{R}^2$, we use $-\pmb{\ell}$ to denote the oriented line antiparallel to $\pmb{\ell}$ that determines the same points of $\pmb{\ell}$ in $\bb{R}^2$. 
\end{enumerate}	
\end{notation}

We remark that $\ell \in \nexpl(\eta)$ if and only if $-\pmb{\ell} \in \nexpd(\eta)$ or $\pmb{\ell} \in \nexpd(\eta)$ (see \cite{van} for more details). As a consequence of a strong condition on the complexity function, in \cite{van} the authors got that $\ell \in \nexpl(\eta)$ if and only if $-\pmb{\ell},\pmb{\ell} \in \nexpd(\eta)$. Our main result (see Theorem \ref{main_thm}) allows us to assume that it holds for low convex complexity configurations. For periodic configurations we have the following result:

\begin{proposition}[Colle and Garibaldi \cite{colle}]\label{pps_par_antipar_exp}
Let $\eta \in \mathcal{A}^{\bb{Z}^2}$ be a periodic configuration. Then $\ell \in \nexpl(\eta)$ if and only if the antiparallel oriented lines $-\pmb{\ell},\pmb{\ell} \in \nexpd(\eta)$.
\end{proposition}

One of the most important recent results related to Nivat's conjecture is due to Jarkko Kari and Etienne Moutot. Using the algebraic viewpoint, they proved the following theorem:

\begin{theorem}[Kari and Moutot \cite{Moutot19}]\label{EtienneMainThm}
Let $\eta \in \mathcal{A}^{\bb{Z}^2}$,\! with \(\mathcal{A} \subset \bb{Z}\), be a configuration with a non-trivial annihilator. Then there exists a configuration \(x \in X_{\eta}\) such that, for every line $\ell \subset \bb{R}^2$ through the origin, $-\pmb{\ell} \not\in \nexpd(x)$ whenever $\pmb{\ell} \not\in \nexpd(x)$.
\end{theorem}

\subsection{The algebraic viewpoint}

Classically in symbolic dynamics configurations are understood as elements of $\mathcal{A}^{\bb{Z}^d}$, but the symbols in the alphabet $\mathcal{A}$ do not matter. From now on, we will suppose $\mathcal{A} \subset R$, where $R$ is $\bb{Z}$ or some finite filed. Following the algebraic viewpoint introduced by Kari and Szabados \cite{KariSzabados}, we may represent any configuration $\eta \in R^{\bb{Z}^d}$ as a formal power series in $d$ variables $x_1, \ldots, x_d$ with coefficients in $R$, i.e., as an element of \[R[[X^{\pm 1}]] = \left\{\sum a_gX^g : g \in \bb{Z}^d, \ a_g \in R\right\},\] where $g = (g_1,\ldots, g_d) $ and $X^g$ is a shorthand for $x_1^{g_1} \cdots x_d^{g_d}$. Here, a formal power series is a representation of a configuration on \(R^{\mathbb{Z}^d}\), but it is important to keep in mind that a formal power series may represent a configuration defined on an infinite alphabet if \(R = \bb{Z}\). Let \(R[X^{\pm 1}]\) denote the Laurent polynomials in \(d\) variables with coefficients in $R$.

Let $\eta \in \mathcal{A}^{\bb{Z}^d}$, with $\mathcal{A} \subset R$, be a configuration. A Laurent polynomial $\varphi(X) = a_1X^{u_1}+\cdots+a_nX^{u_n}$, with $a_i \in R$ and $u_i \in \bb{Z}^d$, \emph{annihilates $\eta$} (in the sense that $\varphi \eta = 0$) if and only if 
\begin{equation}\label{eq_annihilator}
a_1\eta_{-u_1+g}+\cdots+a_n\eta_{-u_n+g} = 0 \quad \forall g \in \bb{Z}^d.
\end{equation}  
We remark that the operations in (\ref{eq_annihilator}) are the binary operations of $(R,+,\cdot)$, the multiplicative and additive identities on $R$ are denoted by $1$ and $0$, respectively, and, for any $x,y \in R$, $x-y := x+(-y)$, where $-y \in R$ denotes the inverse additive of $y$. The set of Laurent polynomials $\varphi \in R[X^{\pm 1}]$ that annihilates $\eta \in \mathcal{A}^{\bb{Z}^d}$\! is de\-no\-ted by $\ann_R(\eta)$. In this algebraic setting, $\eta \in R[[X^{\pm 1}]]$ is periodic of period $h \in \bb{Z}^d$ if and only if $(X^{ h}-1)\eta = 0$.

Kari and Szabados proved the following theorem:

\begin{theorem}[Kari and Szabados \cite{KariSzabados}]\label{theorKS}
Let $\eta \in \mathcal{A}^{\bb{Z}^d}$, with $\mathcal{A} \subset \bb{Z}$, be a configuration.
\begin{enumerate}[(i)]\setlength{\itemsep}{6pt}
	\item If $\eta$ has low complexity, then there exist vectors $h_1,\ldots,h_m \in \bb{Z}^d$ in pairwise distinct directions such that $(X^{h_1}-1) \cdots (X^{h_m}-1) \in \ann_{\bb{Z}}(\eta)$.
	\item If $(X^{h_1}-1) \cdots (X^{h_m}-1) \in \ann_{\bb{Z}}(\eta)$, where $h_1,\ldots,h_m \in \bb{Z}^d$ are vectors in pairwise distinct directions, then there exist periodic configurations $\eta_1,\ldots,\eta_m$ $\in \bb{Z}[[X^{\pm 1}]]$, with $h_i$ a period for $\eta_i$, such that $\eta = \eta_1+\cdots+\eta_m$.   
\end{enumerate} 
\end{theorem} 

The configurations $\eta_i \in \bb{Z}[[X^{\pm 1}]]$ in the Theorem \ref{theorKS} may be defined on infinite alphabets (see \cite[Example 17]{KariSzabados}).

Let $\eta \in \mathcal{A}^{\bb{Z}^d}$ and suppose $\eta_1, \ldots, \eta_m \in R[[X^{\pm 1}]]$ are periodic configurations such that $\eta = \eta_1+ \cdots+\eta_m$. We call $\eta = \eta_1+ \cdots+\eta_m$ a \emph{$R$-periodic decomposition}. If $m \leq n$ for every $R$-pe\-ri\-o\-dic de\-com\-po\-si\-tion $\eta = \vartheta_1+ \cdots+\vartheta_{n}$, we call $\eta = \eta_1+\cdots+\eta_m$ a \emph{$R$-mi\-ni\-mal periodic decomposition} and the number $m$ is called the \emph{order of $\eta$}. 

If \(\eta\) is a non-periodic configuration and $\eta = \eta_1+ \cdots+\eta_m$ is a \(R\)-minimal periodic decomposition, then any two periods for $\eta_i$ and $\eta_j$, with $i \neq j$, are in distinct direc- tions.

Due to Theorem \ref{theorKS}, to prove Nivat's conjecture it is enough to prove the following conjecture:

\begin{conjecture}\label{second_mainconj}
Let $\eta \in \mathcal{A}^{\bb{Z}^2}$, with $\mathcal{A} \subset \bb{Z}$, and suppose $\eta = \eta_1+\cdots+\eta_m$ is a $\bb{Z}$-minimal periodic decomposition with order \(m \geq 2\). Then $\eta$ does not have low convex complexity.
\end{conjecture} 

We remark that Conjecture \ref{second_mainconj} holds for configurations with order $2$:

\begin{theorem}[Szabados \cite{Szabados}]\label{main_theor_szabados}
Let $\eta \in \mathcal{A}^{\bb{Z}^2}$, with $\mathcal{A} \subset \bb{Z}$, and suppose $\eta = \eta_1+\eta_2$ is a $\bb{Z}$-minimal periodic decomposition. Then $\eta$ does not have low convex complexity.
\end{theorem}


If $\eta = \eta_1+\cdots+\eta_m$ is a $\bb{Z}$-minimal periodic decomposition, it is easy to see that every nonexpansive line on $X_{\eta}$ contains a period for some $\eta_i$, with $1 \leq i \leq m$ (see Remark~\ref{rem_minperdecomandperiods}). In his Ph.D. thesis \cite{Szabadosthesis}, Michal Szabados conjectured that the conversely also holds:

\begin{conjecture}[Szabados]
Let $\eta \in \mathcal{A}^{\bb{Z}^2}$, with $\mathcal{A} \subset \bb{Z}$, be a configuration. If $\eta$ is not fully periodic and $\eta = \eta_1+\cdots+\eta_m$ is a $\bb{Z}$-minimal periodic decomposition, then, for each $1 \leq i \leq m$, a line $\ell \subset \bb{R}^2$ contains a period for $\eta_i$ if and only if $\ell \in \nexpl(\eta)$.
\end{conjecture}

\subsection{Our contributions}

\begin{theorem}\label{secondary_thm}
Let $\eta \in \mathcal{A}^{\bb{Z}^2}$, with $\mathcal{A} \subset \bb{Z}$, be a configuration with a non-trivial annihilator. If $-\pmb{\ell} \not\in \nexpd(\eta)$ or $\pmb{\ell} \not\in \nexpd(\eta)$ for all lines $\ell \subset \bb{R}^2$ through the origin, then $\eta$ is fully periodic. 
\end{theorem}

\begin{theorem}\label{main_thm}
Let $\eta \in \mathcal{A}^{\bb{Z}^2}$, with $\mathcal{A} \subset \bb{Z}$, be a low convex complexity con\-fig\-u\-ra\-tion. If \(\eta\) is not fully periodic and all non-periodic configurations in $X_{\eta}$ have the same order, then
\begin{enumerate}[(i)]\setlength{\itemsep}{6pt}
	\item Szabados's conjecture holds;
	\item $\ell \in \nexpl(\eta)$ if and only if the antiparallel oriented lines $-\pmb{\ell},\pmb{\ell} \in \nexpd(\eta)$.
\end{enumerate}
\end{theorem}

In the previous theorem, if \(\eta\) is periodic, then \(X_{\eta}\) does not have any non-pe\-ri\-o\-dic con\-fig\-u\-ra\-tion and so by emptiness all non-periodic configurations in $X_{\eta}$ have the same order. 

We remark that the properties (i) and (ii) in the previous theorem hold for low convex complexity configurations with order \(3\) (see Corollary \ref{SzabadosConjem=3}).    

As an application of the previous theorem, we have the following result:

\begin{corollary}\label{main_cor}
Nivat's conjecture holds for low convex complexity configurations if and only if it holds for low convex complexity configurations $\eta \in \mathcal{A}^{\bb{Z}^2}$, with $\mathcal{A} \subset \bb{Z}$, such that
\begin{enumerate}[(i)]\setlength{\itemsep}{6pt}
	\item Szabados's conjecture holds;
	\item $\ell \in \nexpl(\eta)$ if and only if the antiparallel oriented lines $-\pmb{\ell},\pmb{\ell} \in \nexpd(\eta)$.
\end{enumerate}
\end{corollary}
\begin{proof}
The idea is to argue that Conjecture~\ref{second_mainconj} holds if Nivat's conjecture holds for any low convex complexity configuration satisfying conditions (i) and (ii). 

Suppose Nivat's conjecture holds for any low convex complexity configuration satisfying conditions (i) and (ii). Let \(\vartheta \in \mathcal{A}^{\bb{Z}^2}\), with \(\mathcal{A} \subset \bb{Z}\), be a configuration and suppose \(\vartheta = \vartheta_1 +\cdots+\vartheta_m\) is a \(\bb{Z}\)-minimal periodic decomposition with order \(m \geq 2\). Since Conjecture~\ref{second_mainconj} holds for \(m = 2\) (see Theorem~\ref{main_theor_szabados}), we may assume \(m \geq 3\). Suppose, by induction hypothesis, that Conjecture~\ref{second_mainconj} holds for configurations with order $n \leq m-1$. To approach the inductive step, suppose, by contradiction, that \(\vartheta\) has low convex complexity. Hence, as by induction hypothesis all non-pe\-ri\-o\-dic configurations in $\overline{Orb\,(\vartheta)}$ have the same order $m$, Theorem~\ref{main_thm} implies that 
\begin{enumerate}[(i)]\setlength{\itemsep}{6pt}
	\item Szabados's conjecture holds;
	\item $\ell \in \nexpl(\vartheta)$ if and only if the antiparallel oriented lines $-\pmb{\ell},\pmb{\ell} \in \nexpd(\vartheta)$.
\end{enumerate} 
Therefore, our initial assumption implies that \(\varphi\) is periodic, which is con\-tra\-dic\-tion. So the inductive step is done and Conjecture~\ref{second_mainconj} holds.
\end{proof}

The global strategy for the proof of Theorem \ref{main_thm} consists  basically in showing the existence of a non-periodic accumulation point that is fully periodic on an un\-bound\-ed region, and then in arguing that the boundary geometry of this region has a rigid structure (see Subsection \ref{subsec} for a detailed picture). 

The rest of the paper is organized as follows: in Section~\ref{sec2}, we review key concepts and results. In Section~\ref{sec3}, we prove preliminary results and Theorem~\ref{secondary_thm} and, in Section~\ref{sec4}, we prove preliminary results and Theorem~\ref{main_thm}. 

\section{Background}
\label{sec2}
In this section, we will revisit notions, known results or immediate applications of known results.   

\subsection{Geometric notations}

Let $\mathcal{S} \subset \bb{Z}^2$ be a convex set. A point $g \in \mathcal{S}$ is called a \emph{vertex} of $\mathcal{S}$ when $\mathcal{S} \backslash \{g\}$ is still a convex set. If $\conv(\mathcal{S})$ has positive area, an edge of the convex polygon $\conv(\mathcal{S})$ is called an \emph{edge} of $\mathcal{S}$. We use $V(\mathcal{S})$ and $E(\mathcal{S})$ to denote, respectively, the sets of vertices and edges of $\mathcal{S}$. 

If $\mathcal{S} \subset \bb{Z}^2$ is a convex set (possibly infinite) such that $\conv(\mathcal{S})$ has positive area, our standard convention is that the boundary of $\conv(\mathcal{S})$ is positively oriented. With this convention, each edge $w \in E(\mathcal{S})$ inherits a natural orientation from the boundary of $\conv(\mathcal{S})$. If \(|E(\mathcal{S})| < \infty\), our convention endows each \(w \in E(\mathcal{S})\) with a well-defined successor edge \(w_s \in E(\mathcal{S})\) and a well-defined predecessor edge \(w_p \in E(\mathcal{S})\). We use \(w_p \prec w\) to indicate that \(w_p \in E(\mathcal{S})\) is the predecessor edge of \(w \in E(\mathcal{S})\).   

\begin{definition}\label{retabase}
Let $\pmb{\ell} \subset \bb{R}^2$ be an oriented line and suppose $\mathcal{S} \subset \bb{Z}^2$ is a non-empty set such that $\mathcal{S}+u \subset \mathcal{H}(\pmb{\ell})$ for some $u \in \bb{Z}^2$. The \emph{support line of $\mathcal{S}$} determined by $\pmb{\ell}$, denoted by $\pmb{\ell}_{\mathcal{S}}$, is defined as the oriented line $\pmb{\ell}'$ parallel to $\pmb{\ell}$ such that $\mathcal{S} \subset \mathcal{H}(\pmb{\ell}')$ and $\pmb{\ell}' \cap \mathcal{S} \neq \emptyset$.
\end{definition}

Let $\pmb{\ell} \subset \bb{R}^2$ be an oriented line. If $\mathcal{S} \subset \bb{Z}^2$ is a non-empty, finite, convex set, then $\pmb{\ell}_{\mathcal{S}}$ is defined and either $\mathcal{S} \cap \pmb{\ell}_{\mathcal{S}}$ is a vertex of $\mathcal{S}$ or $\conv(\mathcal{S} \cap \pmb{\ell}_{\mathcal{S}})$ is an edge of $\mathcal{S}$ parallel to $\pmb{\ell}$ (see Figure \ref{fig1}). The support line of \(\mathcal{S}\) determined by \(-\pmb{\ell}\), the oriented line \(-\pmb{\ell}_{\mathcal{S}}\), is also defined and \(-\pmb{\ell}_{\mathcal{S}} \cap \pmb{\ell}_{\mathcal{S}} = \emptyset\) if \(\mathcal{S}\) has positive area. 

\begin{figure}[ht]
\centering
\def\svgwidth{6cm}
\includegraphics[width=6cm]{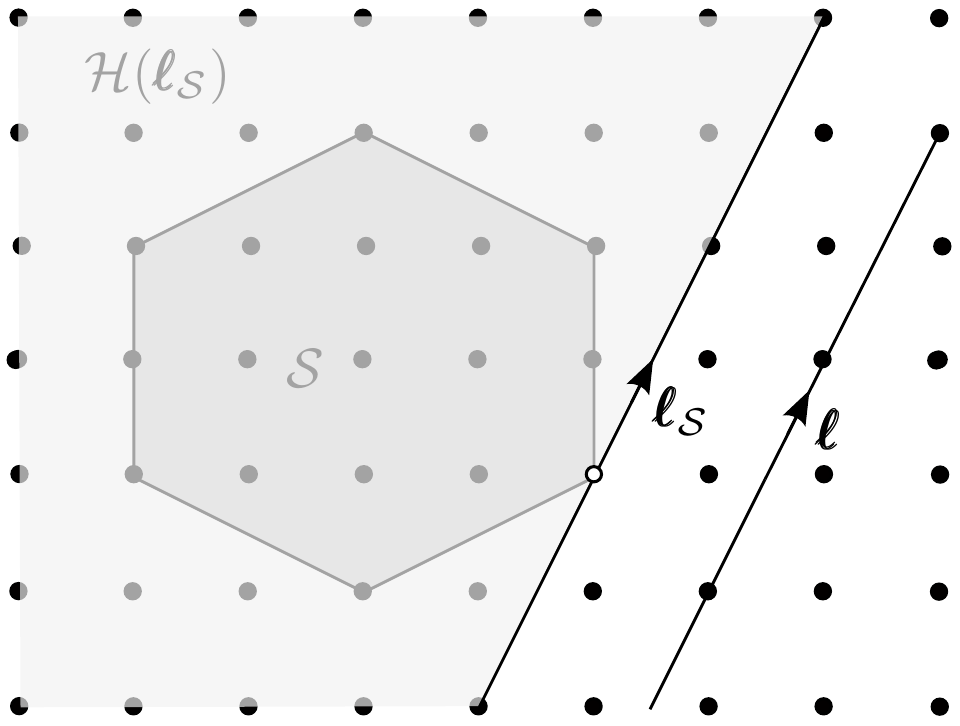}\\
\caption{The set $\mathcal{S}$ and the oriented lines $\pmb{\ell}$ and $\pmb{\ell}_{\mathcal{S}}$.} 
\label{fig1}
\end{figure}

\subsection{Generating sets}
The notion of generating set was developed in \cite{van} and studied in a more general setting on \cite{frankskra}. 

\begin{definition}
Let $\eta \in \mathcal{A}^{\bb{Z}^2}$ be a configuration and suppose $\mathcal{S} \subset \bb{Z}^2$ is a finite set. A point $g \in \mathcal{S}$ is said to be \emph{$\eta$-generated by $\mathcal{S}$} if $P_{\eta}(\mathcal{S}) = P_{\eta}(\mathcal{S} \backslash \{g\})$. A non-empty, finite, convex subset of $\bb{Z}^2$ for which each vertex is $\eta$-generated is called an \emph{$\eta$-generating set}.
\end{definition}

As highlighted in \cite{van}, the geometry of an \(\eta\)-generating set is related to the one-sided nonexpansive directions on \(X_{\eta}\). Actually, as an immediate corollary from Lemma \ref{lem_genset_noedge_expas}, which is similar to Lemma 2.13 of \cite{van}, one has that all one-sided nonexpansive directions on $X_{\eta}$ are parallel to some edge of every $\eta$-generating set whose convex hull has positive area. In particular, lines in $\nexpl(\eta)$ or oriented lines in \(\nexpd(\eta)\) are rational, where by \emph{rational} we mean a line or an oriented line in $\bb{R}^2$ with a rational slope.

\begin{lemma}\label{lem_genset_noedge_expas}
Let $\eta \in \mathcal{A}^{\bb{Z}^2}$\! and suppose $\pmb{\ell} \subset \bb{R}^2$ is an oriented line through the ori\-gin and $\mathcal{S} \subset \bb{Z}^2$ is a finite, convex set such that $\mathcal{S} \cap \pmb{\ell}_{\mathcal{S}} = \{g_0\}$ is $\eta$-generated by $\mathcal{S}$. Then $\pmb{\ell} \not\in \nexpd(\eta)$.
\end{lemma}
\begin{proof}
It is straightforward from definitions.
\end{proof}

The next lemma shows, in particular, that every low convex complexity configuration has a generating set:

\begin{lemma}[Cyr and Kra \cite{van}]\label{lemma_generatingset}
Let $\eta \in \mathcal{A}^{\bb{Z}^2}$ be a low convex complexity con\-fig\-u\-ra\-tion. Then there exists an $\eta$-generating set $\mathcal{S} \subset \mathbb{Z}^2$ such that $$P_{\eta}(\mathcal{S})-P_{\eta}(\mathcal{S} \backslash \pmb{\ell}_{\mathcal{S}}) \leq |\mathcal{S} \cap \pmb{\ell}_{\mathcal{S}}|-1$$ for all oriented lines $\pmb{\ell} \subset \bb{R}^2$.
\end{lemma}

Let $\varphi(X) = \sum_{i} a_iX^{u_i}$, with $a_i \in R$ and $u_i \in \bb{Z}^2$, be a Laurent polynomial. The \emph{support} of $\varphi$ is defined as $$\supp(\varphi) := \{u_i \in \bb{Z}^2 : a_i \neq 0\}.$$ The \emph{reflected convex support of $\varphi$} is defined as $\mathcal{S}_{\varphi} := \conv(-\supp(\varphi)) \cap \bb{Z}^2$.

\begin{lemma}[Szabados \cite{Szabados}]\label{lemma_supportgenerating}
Let $\eta \in \mathcal{A}^{\bb{Z}^2}$\! be a configuration. If the Laurent polynomial $\varphi \in \ann_R(\eta)$, then $\mathcal{S}_{\varphi}$ is an $\eta$-generating set.	
\end{lemma}
\begin{proof}
Given $g \in \bb{Z}^2$ and $u_0 \in \supp(\varphi)$, the knowledge of the configuration $\eta$ on $-u+g$ for $u \in \supp(\varphi)$, with $u \neq u_0$, determines the configuration $\eta$ on $-u_0+g$ by applying equation (\ref{eq_annihilator}). This means that $-u_0 \in -\supp(\varphi)$ is $\eta$-generated by $-\supp(\varphi)$.  
\end{proof}

\begin{lemma}[Szabados \cite{Szabados}]\label{lem_periods_annihilator}
Let $\eta \in \mathcal{A}^{\bb{Z}^2}$ be a configuration and suppose $\varphi(X) = (X^{h_1}-1) \cdots (X^{h_m}-1) \in \ann_R(\eta)$, where \(h_1, \ldots, h_m \in \bb{Z}^2\) are vectors in pairwise distinct directions. If $\ell \in \nexpl(\eta)$, then \(\ell\) contains some vector $h_i$, with $1 \leq i \leq m$.
\end{lemma}
\begin{proof}
Note that
\begin{equation*}
\supp(\varphi) = \{(0,0)\} \cup \{h_{i_1}+\cdots+h_{i_r} : 1 \leq i_1 < \cdots < i_r \leq m, \ 1 \leq r \leq m\}.
\end{equation*}
Then \(|E(\mathcal{S}_{\varphi})| = 2m\) and, for each \(w \in E(\mathcal{S}_{\varphi})\), \(w\) is either parallel or antiparallel to some vector \(h_i\), with \(1 \leq i \leq m\). Therefore, since $\mathcal{S}_{\varphi}$ is an \(\eta\)-generating set and $\ell \in \nexpl(\eta)$ if and only if $-\pmb{\ell} \in \nexpd(\eta)$ or $\pmb{\ell} \in \nexpd(\eta)$, the result follows from Lemma~\ref{lem_genset_noedge_expas} applied to \(\mathcal{S}_{\varphi}\).
\end{proof}

Of course, we may consider a one-sided nonexpansive direction instead of a non- expansive line in the previous lemma. 

\begin{remark}\label{rem_minperdecomandperiods}
If $\eta = \eta_1+\cdots+\eta_m$ is a $R$-periodic decomposition and $\ell \in \nexpl(\eta)$, then $\ell$ contains a period for some $\eta_i$, with $1 \leq i \leq m$. Indeed, if $h_i \in \bb{Z}^2$ denotes a period for $\eta_i$, the Laurent polynomial $\varphi(X) = (X^{h_1}-1) \cdots (X^{h_m}-1) \in \ann_R(\eta)$.
\end{remark}

To conclude this section, we highlight that, according to Corollary 24 of \cite{KariSzabados}, the ge\-om\-e\-try of $\mathcal{S}_{\psi}$ has a rigid structure if \(\psi \in \ann_{\bb{Z}}(\eta)\):    

\begin{proposition}[Kari and Szabados \cite{KariSzabados}]\label{prop_geom_convexset}
Let $\eta \in \mathcal{A}^{\bb{Z}^2}$, with $\mathcal{A} \subset \bb{Z}$, be a non-periodic configuration. Suppose $\eta = \eta_1+\cdots+\eta_m$ is a $\bb{Z}$-minimal periodic decomposition, $h_i \in \bb{Z}^2$ is a period for $\eta_i$, with \(1 \leq i \leq m\), and $\varphi(X) = (X^{h_1}-1) \cdots (X^{h_{m}}-1)$. If $\psi \in \ann_{\bb{Z}}(\eta)$, then, for every edge $w \in E(\mathcal{S}_{\varphi})$, there exists an edge $w' \in E(\mathcal{S}_{\psi})$ parallel to $w$.
\end{proposition}
 
Pro\-po\-si\-tion~\ref{prop_geom_convexset} provides a bound for the number of periodic configurations of a $\bb{Z}$-minimal periodic decomposition:

\begin{remark}
Let $\eta \in \mathcal{A}^{\bb{Z}^2}$\!, with $\mathcal{A} \subset \bb{Z}$, and suppose there exist $n,k \in \bb{N}$ such that $P_{\eta}(n,k) \leq nk$. If $\eta = \eta_1+\cdots+\eta_m$ is a $\bb{Z}$-minimal periodic decomposition, then $m \leq \min\{n,k\}$. In fact, following the approach of \cite{kari}, there exist a Laurent polynomial $\sigma \in \bb{Z}[X^{\pm 1}]$, with \(-\supp(\sigma) \subset I_n \times I_k\), and \(c \in \bb{Z}\) such that $\sigma \eta = c$. Since $\mathcal{S}_{\sigma}$ is contained in $I_n \times I_k$, then $\mathcal{S}_{\sigma}$ has at most $2\min\{n,k\}$ edges. For an appropriated $u \in \bb{Z}^2$, the reflected convex support of $(X^{u}-1)$ $\sigma(X) \in \ann_{\bb{Z}}(\eta)$ has at most $2\min\{n,k\}$ edges. So Proposition~\ref{prop_geom_convexset} applied to $\psi(X) = (X^{u}-1)\sigma(X)$ allows us to conclude that $m \leq \min\{n,k\}$.
\end{remark}

\subsection{Periodicity}

In this section, we state results related to periodicity. To begin, we revisit the notion of ambiguity of \cite{van}. Let $\pmb{\ell} \subset \bb{R}^2$ be an oriented line through the origin and suppose that $\mathcal{S} \subset \bb{Z}^2$ is a non-empty, finite, convex set. For each $\gamma \in \pat(\mathcal{S} \backslash \pmb{\ell}_{\mathcal{S}},\eta)$, we define
\begin{equation*}\label{eq_dfn_cont}
N_{\mathcal{S}}(\pmb{\ell},\gamma) := \left|\{\gamma' \in \pat(\mathcal{S},\eta) : \gamma'\sob{\mathcal{S} \backslash \pmb{\ell}_{\mathcal{S}}} = \gamma\}\right|.
\end{equation*} 
Notice that $N_{\mathcal{S}}(\pmb{\ell},\gamma) = 1$ means that 
$\gamma'\sob{\mathcal{S} \backslash \pmb{\ell}_{\mathcal{S}}} = \gamma = \gamma''\sob{\mathcal{S} \backslash \pmb{\ell}_{\mathcal{S}}}$ implies $\gamma' = \gamma''$ for any $\gamma',\gamma'' \in \pat(\mathcal{S},\eta)$. Let $\pmb{\ell} \subset \bb{R}^2$ be a rational oriented line through the origin.  A configuration $x \in X_{\eta}$ is said to be \emph{$(\pmb{\ell},\mathcal{S})$-ambiguous} if, 
\begin{equation*}\label{ambiguous}
\forall \  \mathcal{S} \backslash \pmb{\ell}_{\mathcal{S}}\textrm{-pattern} \ \gamma = (T^{t\vec{v}_{\pmb{\ell}}}x)\sob{\mathcal{S} \backslash \pmb{\ell}_{\mathcal{S}}}, \ t \in \bb{Z}, \ \ \text{one has} \ N_{\mathcal{S}}(\pmb{\ell},\gamma) > 1,  
\end{equation*}
where $\vec{v}_{\pmb{\ell}} \in \pmb{\ell} \cap \bb{Z}^2$ denotes the non-zero vector parallel to $\pmb{\ell}$ of minimum norm. A configuration $x \in X_{\eta}$ is said to be \emph{$(\pmb{\ell},\mathcal{S},\pm)$-semi-ambiguous} if there exists $\tau \in \bb{Z}$ such that,
\begin{equation*}\label{semi-ambiguous}
\forall \  \mathcal{S} \backslash \pmb{\ell}_{\mathcal{S}}\textrm{-pattern} \ \gamma = (T^{\pm t\vec{v}_{\pmb{\ell}}}x)\sob{\mathcal{S} \backslash \pmb{\ell}_{\mathcal{S}}}, \ t \geq \tau, \ \ \text{one has} \ N_{\mathcal{S}}(\pmb{\ell},\gamma) > 1.
\end{equation*}
  
If $x \in X_{\eta}$ is an $(\pmb{\ell},\mathcal{S},+)$-semi-ambiguous configuration, then each \(\mathcal{S} \backslash \pmb{\ell}_{\mathcal{S}}\)-pattern of \(x\) arising on the half strip \(\{g+t\vec{v}_{\pmb{\ell}} : g \in \mathcal{S} \backslash \pmb{\ell}_{\mathcal{S}}, \ t \geq \tau\}\) extends to at least two distinct \(\mathcal{S}\)-patterns of \(\eta\). Similarly, if $x \in X_{\eta}$ is an $(\pmb{\ell},\mathcal{S},-)$-semi-ambiguous configuration, then each \(\mathcal{S} \backslash \pmb{\ell}_{\mathcal{S}}\)-pattern of \(x\) arising on the half strip \(\{g-t\vec{v}_{\pmb{\ell}} : g \in \mathcal{S} \backslash \pmb{\ell}_{\mathcal{S}}, \ t \geq \tau\}\) extends to at least two distinct \(\mathcal{S}\)-patterns of \(\eta\). It is not difficult to see that there always exist $(\pmb{\ell},\mathcal{S})$-ambiguous configurations in $X_{\eta}$ if $\pmb{\ell} \in \nexpd(\eta)$ and $\mathcal{S} \subset \bb{Z}^2$ is an $\eta$-gene\-ra\-ting set.

The next result is a summary of two results of \cite{van}, namely, Lemma 2.24 and a piece of the proof of Proposition 4.8.

\begin{proposition}[Cyr and Kra \cite{van}]\label{prop_CyrKra}
Let $\eta \in \mathcal{A}^{\bb{Z}^2}$ be a configuration and suppose $\pmb{\ell} \in \nexpd(\eta)$. If there exists an $\eta$-generating set $\mathcal{S} \subset \bb{Z}^2$, with $|\mathcal{S} \cap \pmb{\ell}_{\mathcal{S}}| \leq |\mathcal{S} \cap -\pmb{\ell}_{\mathcal{S}}|$, such that \(P_{\eta}(\mathcal{S})-P_{\eta}(\mathcal{S} \backslash \pmb{\ell}_{\mathcal{S}}) \leq |\mathcal{S} \cap \pmb{\ell}_{\mathcal{S}}|-1\), then, for any $(\ell,\mathcal{S})$-ambiguous con\-fig\-u\-ra\-tion $x \in X_{\eta}$, the restriction of $x$ to the half plane \(\mathcal{H}(-\pmb{\ell}_{\mathcal{S}})\) is periodic with period parallel to \(\pmb{\ell}\).
\end{proposition}

Of course, a configuration is periodic with period parallel to \(\pmb{\ell}\) if and only if it is periodic with period parallel to \(-\pmb{\ell}\). 

We will need a definition of periodicity for configurations restricted to sets:

\begin{definition}
Let $\eta \in \mathcal{A}^{\bb{Z}^2}$ be a configuration and suppose $\mathcal{U} \subset \bb{Z}^2$ is a non-empty set. We say that \(\eta\sob{\mathcal{U}}\), the restriction of $\eta$ to $\mathcal{U}$, is \emph{periodic} if there exists a vector $h \in (\bb{Z}^2)^*$, called \emph{period of \(\eta\sob{\mathcal{U}}\)}, such that $\eta_{g+h} = \eta_{g}$ for all $g \in \mathcal{U} \cap (\mathcal{U}-h)$. If \(\eta\sob{\mathcal{U}}\) has two periods linearly independents over \(\bb{R}^2\), we say that \(\eta\sob{\mathcal{U}}\) is fully periodic.  
\end{definition}

The next result is a particular case of \cite[Lemma 39]{KariSzabados}.

\begin{proposition}[Kari and Szabados \cite{KariSzabados}]\label{prop_dec_period_half_plane}
Let $\eta \in \mathcal{A}^{\bb{Z}^2}$ be a configuration and sup\-pose  $\eta = \eta_1+\cdots+\eta_m$ is a $R$-periodic decomposition, $h_i \in \bb{Z}^2$ is a period for $\eta_i$ and $h_1, \ldots, h_m \in \bb{Z}^2$ vectors in pairwise distinct directions. If $\eta\sob{\mathcal{H}(\pmb{\ell})}$ is pe\-ri\-o\-dic with period parallel to some oriented line $\pmb{\ell} \subset \bb{R}^2$, then $\eta$ is periodic with period parallel to \(\pmb{\ell}\).
\end{proposition}
\begin{proof}
Translating $\eta$ if necessary, we may assume that $\pmb{\ell}$ through the origin. The case where $\pmb{\ell} \not\in \nexpd(\eta)$ is straightforward. Indeed, let $h \in \bb{Z}^2$ be a period for $\eta \sob{\mathcal{H}(\pmb{\ell})}$ parallel to \(\pmb{\ell}\). Since the restrictions of $\eta$ and $T^h\eta$ to the half plane $\mathcal{H}(\pmb{\ell})$ coincide, by expansiveness, one has $T^h\eta = \eta$, which concludes this case. From now on, we will suppose $\pmb{\ell} \in \nexpd(\eta)$. In particular, as $(X^{h_1}-1) \cdots (X^{h_m}-1) \in \ann_R(\eta)$, from Lemma~\ref{lem_periods_annihilator} we get that $\pmb{\ell}$ contains some vector $h_j$, with $1 \leq j \leq m$. If $m = 1$ there is nothing to argue. Thus we may consider $m \geq 2$ and, renaming if necessary, we may assume $j = m$. Suppose initially that $R$ is a finite field. Since $(X^{h_1}-1) \cdots (X^{h_{m-1}}-1) \in \ann_R(\eta-\eta_m)$ and \(\pmb{\ell}\) does not contain any of the vectors $h_1, \ldots, h_{m-1} \in \bb{Z}^2$, from Lemma~\ref{lem_periods_annihilator} we conclude that $\pmb{\ell} \not\in \nexpd(\eta-\eta_m)$. Hence, as $(\eta-\eta_m)\sob{\mathcal{H}(\pmb{\ell})}$ is periodic with period parallel to $\pmb{\ell}$, by expansiveness results that $\eta-\eta_m$ is periodic with period parallel to $\pmb{\ell}$ and so that $\eta$ is periodic with period parallel to $\pmb{\ell}$. The case where $R = \bb{Z}$ follows from the previous case: changing the alphabet if necessary, let $p \in \bb{N}$ be a prime number such that $\mathcal{A} \subset \bb{Z}_p$ and consider the $\bb{Z}_p$-pe\-ri\-o\-dic decomposition $\eta = \overline{\eta}_1+\cdots+\overline{\eta}_m$, where $(\overline{\eta}_i)_g := (\eta_i)_g \mod p$ for all $g \in \bb{Z}^2$.
\end{proof}

\section{Preliminaries and proof of Theorem \ref{secondary_thm}}
\label{sec3}

\begin{definition}\label{ll'maximalperconf}
Let $\pmb{\ell}, \pmb{\ell}' \subset \bb{R}^2$ be rational oriented lines in distinct directions. A convex set $\mathcal{R} \subset$ $\bb{Z}^2$ (positively oriented) with two semi-infinite edges parallels to $\pmb{\ell}$ and $\pmb{\ell}'$, respectively, is called an \emph{$(\pmb{\ell},\pmb{\ell}')$-region} if \(w \prec \cdots \prec w'\), where \(w \in E(\mathcal{R})\) is the edge parallel to $\pmb{\ell}$ and \(w' \in E(\mathcal{R})\) is the edge parallel to $\pmb{\ell}'$.
\end{definition}

In the previous definition, \(w \prec \cdots \prec w'\) means that, following the orientation of \(\mathcal{R}\), the edge of \(\mathcal{R}\) parallel to $\pmb{\ell}$ comes first than the edge of \(\mathcal{R}\) parallel to $\pmb{\ell}'$ (see Figure~\ref{llmaxreg}). 

\begin{figure}[!htbp]
	\centering
	\begin{minipage}[b]{0.49\linewidth}
		\centering\includegraphics[width=5cm]{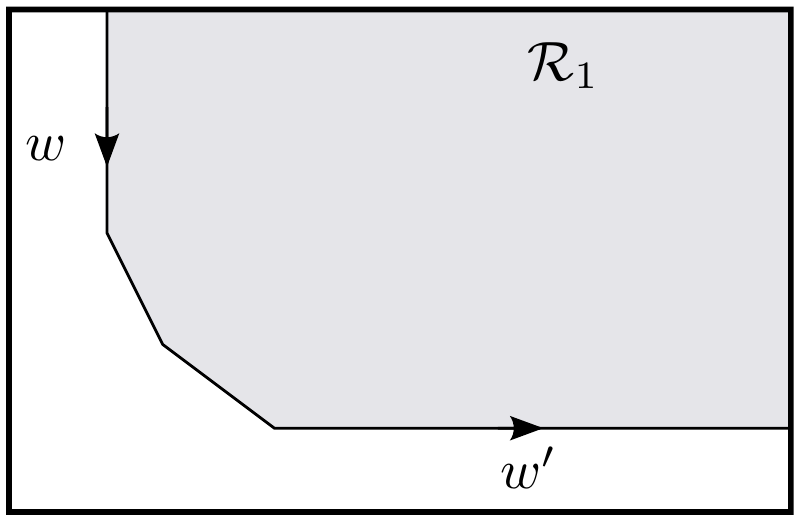}\\
		\small{(A) $\mathcal{R}_1$ is an $(\pmb{\ell},\pmb{\ell}')$-region.}
	\end{minipage} 
	\hfill
	\begin{minipage}[b]{0.49\linewidth}
		\centering\includegraphics[width=5cm]{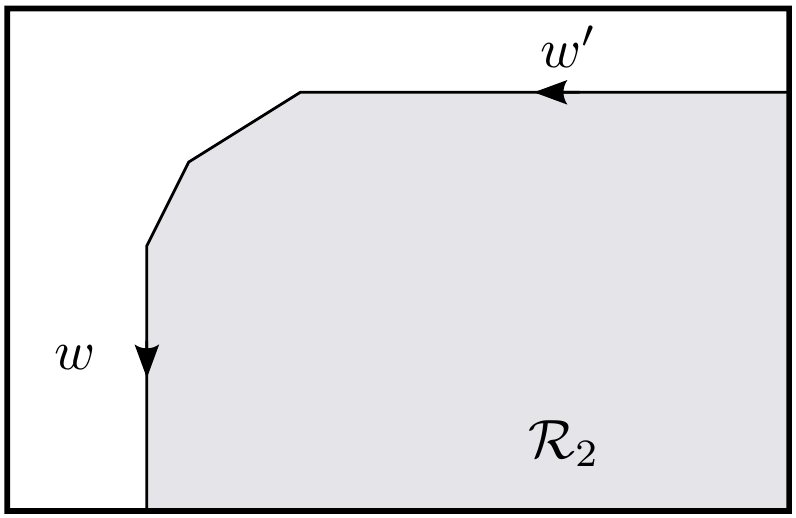}\\
		\small{(B) $\mathcal{R}_2$ is an $(-\pmb{\ell}'\!,\pmb{\ell})$-region.}
	\end{minipage}
	\hfill
	\vspace{0.4cm}
	\begin{minipage}[b]{0.49\linewidth}
		\centering\includegraphics[width=5cm]{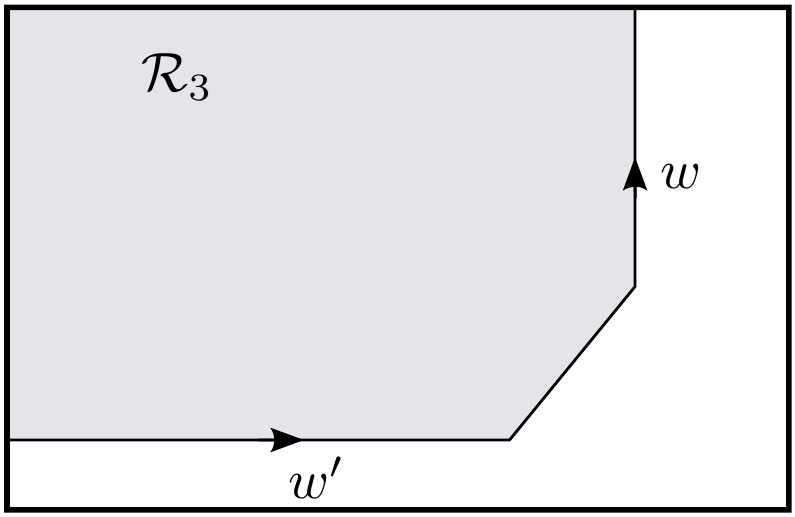}\\		
		\small{(C) $\mathcal{R}_3$ is an $(\pmb{\ell}'\!,-\pmb{\ell})$-region.}
	\end{minipage}
	\hfill
	\begin{minipage}[b]{0.49\linewidth}
		\centering\includegraphics[width=5cm]{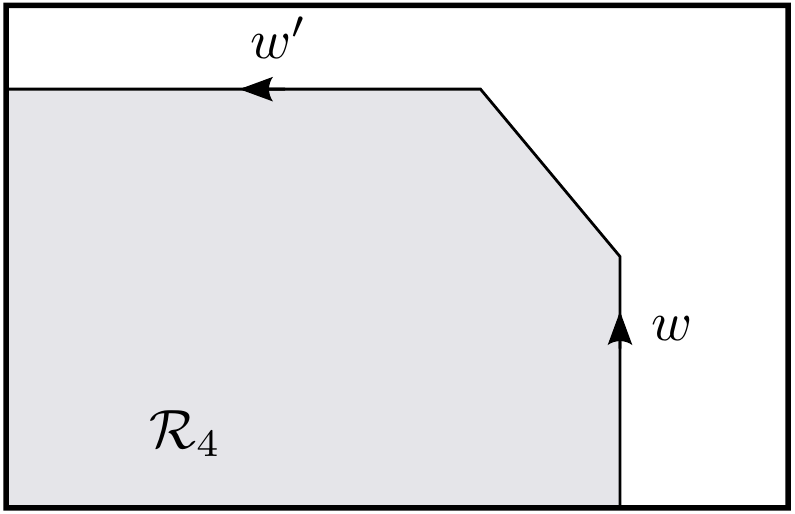}\\
		\small{(D) $\mathcal{R}_4$ is an $(-\pmb{\ell},-\pmb{\ell}')$-region.}
	\end{minipage}
	\caption{Fixed two oriented lines, we have the four types of infinite convex regions. In each case, the semi-infinite edges $w,w' \in E(\mathcal{R}_i)$, $ i = 1, 2, 3, 4 $, are parallel to the corresponding oriented lines.}
	\label{llmaxreg}
\end{figure}

\begin{definition}
Let $\mathcal{U} \subset \bb{Z}^2$ be a finite, convex set such that $\conv(\mathcal{U})$ has positive area. A convex set $\mathcal{T} \subset \bb{Z}^2$ is said to be weakly $E(\mathcal{U})$-enve\-loped if, for every edge $\varpi \in E(\mathcal{T})$, there exists an edge $w \in E(\mathcal{U})$ parallel to $\varpi$ with $|w \cap \mathcal{U}| \leq |\varpi \cap \mathcal{T}|$. A set $\mathcal{T} \subset \bb{Z}^2$ weakly $E(\mathcal{U})$-enveloped where $|E(\mathcal{T})| = |E(\mathcal{U})|$ is said to be $E(\mathcal{U})$-en\-ve\-lo\-ped.
\end{definition}

\begin{notation}\label{def_next_line}
Let $\pmb{\ell} \subset \bb{R}^2$ be a rational oriented line with $\pmb{\ell} \cap \bb{Z}^2 \neq \emptyset$. We use $\pmb{\ell}^{(-)}$\! to denote the oriented line $\pmb{\ell}'  \subset \bb{R}^2$ parallel to $\pmb{\ell}$ closest of $\mathcal{H}(\pmb{\ell})$ such that $\pmb{\ell}' \cap \bb{Z}^2 \neq \emptyset$ and $\mathcal{H}(\pmb{\ell}) \cap \pmb{\ell}' = \emptyset$ (see Figure \ref{fig3}).
\end{notation}

\begin{figure}[ht]
	\centering\includegraphics[width=6cm]{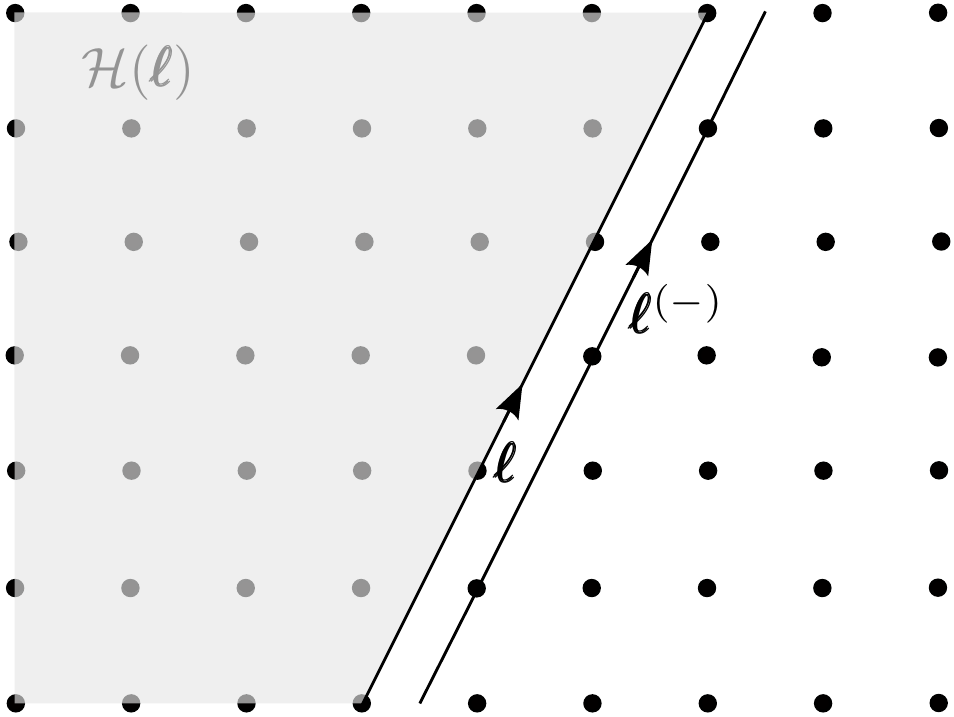}\\
	\caption{The oriented lines $\pmb{\ell}$ and $\pmb{\ell}^{(-)}$ and the half plane $\mathcal{H}(\pmb{\ell})$.}
	\label{fig3}
\end{figure}

\begin{definition}
Let \(\pmb{\ell} \subset \bb{R}^2\) be a rational oriented line and suppose $B \subset \bb{Z}^2$ is a non-empty, finite, convex set. The \emph{half-strip from $B$ in the direction of $\pmb{\ell}$} is defined as \[H_{B}(\pmb{\ell}) := \{g+t\vec{v}_{\pmb{\ell}} \in \bb{Z}^2 : g \in B, \ t \in \bb{Z}_+\}.\]
\end{definition}

For a rational oriented line \(\pmb{\ell} \subset \bb{R}^2\) and a non-empty, finite, convex set \(B \subset \bb{Z}^2\), we may consider the half strips \(H_B(\pmb{\ell})\) and \(H_B(-\pmb{\ell})\) (see Figure \ref{def_halfstrip}).

\begin{figure}[ht]
	\centering\includegraphics[width=12.4cm]{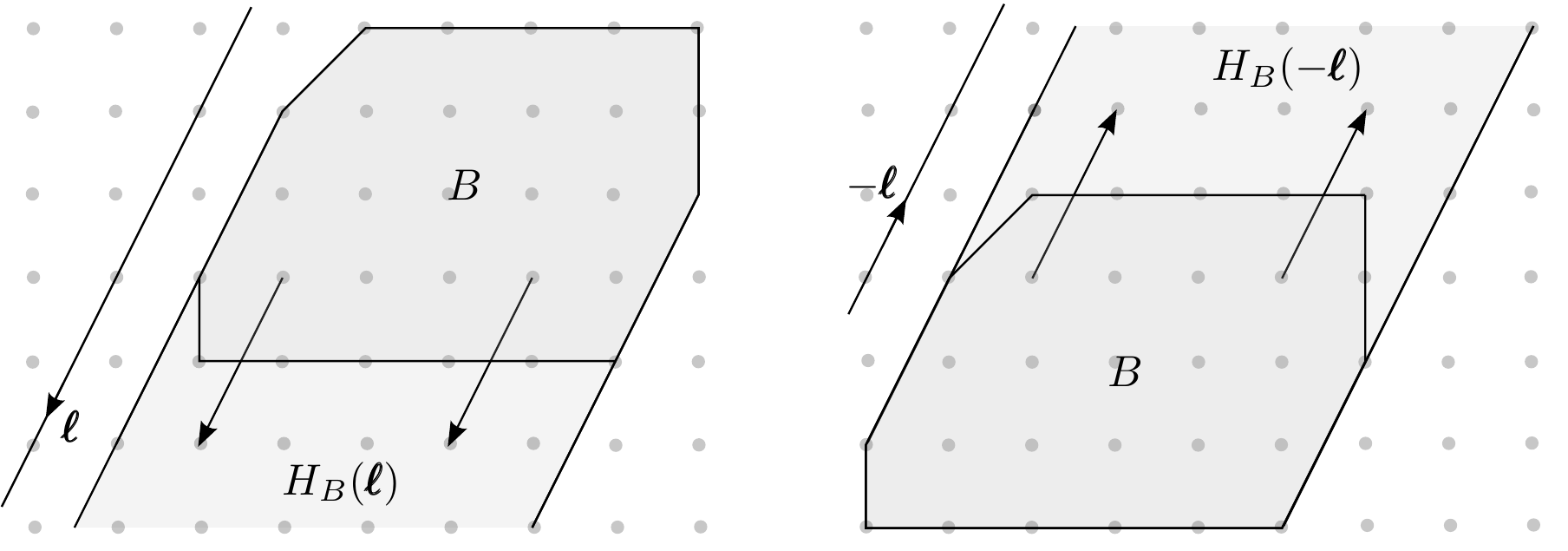}\\
	\caption{The set \(B \subset \bb{Z}^2\), the oriented lines \(\pmb{\ell}\) and \(-\pmb{\ell}\) and the half strips \(H_B(\pmb{\ell})\) and \(H_B(-\pmb{\ell})\).}
	\label{def_halfstrip}
\end{figure}

In the next lemma, \((\pmb{\ell}_{\iota})_{B'} = (\pmb{\ell}_{\iota})_{B}\) means that the support lines of \(B'\) and \(B\) de\-ter\-mi\-ned by \(\pmb{\ell}_{\iota}\) coincide.     

\begin{lemma}\label{tech_lemma}
Let $\eta \in \mathcal{A}^{\bb{Z}^2}$ be a configuration with \(\varphi(X) = (X^{h_1}-1) \cdots (X^{h_m}-1) \in\) \(\ann_{R}(\eta)\), where \(h_1, \ldots, h_m \in \bb{Z}^2\), with \(m \geq 2\), are vectors in pairwise distinct directions. Suppose $\pmb{\ell}_1,\ldots,\pmb{\ell}_{2m} \subset \bb{R}^2$ is an enumeration of the oriented lines through the origin parallels to the edges of $\mathcal{S}_{\varphi}$ where the edge parallel to $\pmb{\ell}_{i+1}$ is a successor of the edge parallel to $\pmb{\ell}_{i}$ and indices are taken modulo $2m$. Given $1 \leq \iota \leq 2m$, suppose \(x_{per} \in X_{\eta}\) is a periodic configuration with period parallel to \(\pmb{\ell}_{\iota}\).
\begin{enumerate}[(i)]\setlength{\itemsep}{5pt}
	\item If for each $E(\mathcal{S}_{\varphi})$-enveloped set $B' \subset \bb{Z}^2$ there exist an $E(\mathcal{S}_{\varphi})$-enveloped set $B \supset B'$, with \((\pmb{\ell}_{\iota})_{B'} = (\pmb{\ell}_{\iota})_{B}\), and $u \in \bb{Z}^2$ such that $$(T^u\eta)\sob{B} = x_{per}\sob{B}, \ \ \text{but} \ \ (T^u\eta)\sob{H_{B}(\pmb{\ell}_{\iota})} \neq x_{per}\sob{H_{B}(\pmb{\ell}_{\iota})},$$ then there exist an \((\pmb{\ell}_{\iota},\pmb{\ell}_J)\)-region \(\hat{A}_{\infty}\), with \(\iota+1 \leq J \leq \iota+m-1\), a configuration \(\vartheta \in X_{\eta}\) and \(\epsilon \in \bb{Z}_+\) such that
	\begin{equation}\label{eq_existence_tau_lem1}
		\vartheta\sob{\hat{A}_{\infty}^{(\epsilon)}} = \hat{x}_{per}\sob{\hat{A}_{\infty}^{(\epsilon)}}, \ \ \textrm{but} \ \ \vartheta\sob{\hat{A}_{\infty}^{(\epsilon+1)}} \neq \hat{x}_{per}\sob{\hat{A}_{\infty}^{(\epsilon+1)}},
	\end{equation}
	where $\hat{x}_{per} := T^{k\vec{v}_{\ell}}x_{per}$ for some \(k \in \bb{Z}_+\), for each $\epsilon \in \bb{Z}_+$, \[\hat{A}_{\infty}^{(\epsilon)} := \{g+t\vec{v}_{\pmb{\ell}_{J-1}} : g \in \hat{A}_{\infty}, \ t \in \bb{Z}_+, \ \dist(g+t\vec{v}_{\pmb{\ell}_{J-1}},\pmb{\ell}_J) \leq d_{\epsilon}\}\] is an \((\pmb{\ell}_{\iota},\pmb{\ell}_J)\)-region and $0 = d_0 < d_1 < \cdots < d_n < \cdots$ is the sequence where, for each $g \in \mathcal{H}(\pmb{\ell}_J)$, there exists $i \in \bb{Z}_+$ such that $\dist(g,\pmb{\ell}_J) = d_i$.

	\item If for each $E(\mathcal{S}_{\varphi})$-enveloped set $B' \subset \bb{Z}^2$ there exist an $E(\mathcal{S}_{\varphi})$-enveloped set $B \supset B'$, with \((\pmb{\ell}_{\iota})_{B'} = (\pmb{\ell}_{\iota})_{B}\), and $u \in \bb{Z}^2$ such that $$(T^u\eta)\sob{B} = x_{per}\sob{B}, \ \ but \ \ (T^u\eta)\sob{H_{B}(-\pmb{\ell}_{\iota})} \neq x_{per}\sob{H_{B}(-\pmb{\ell}_{\iota})},$$ then there exist an \((\pmb{\ell}_J,\pmb{\ell}_{\iota})\)-region \(\hat{A}_{\infty}\), with \(\iota+m+1 \leq J \leq \iota+2m-1\), a configuration \(\vartheta \in X_{\eta}\) and \(\epsilon \in \bb{Z}_+\) such that
	\begin{equation}\label{eq_existence_tau_lem2}
		\vartheta\sob{\hat{A}_{\infty}^{(\epsilon)}} = \hat{x}_{per}\sob{\hat{A}_{\infty}^{(\epsilon)}}, \ \ \textrm{but} \ \ \vartheta\sob{\hat{A}_{\infty}^{(\epsilon+1)}} \neq \hat{x}_{per}\sob{\hat{A}_{\infty}^{(\epsilon+1)}},
	\end{equation}
	where $\hat{x}_{per} := T^{k\vec{v}_{\ell}}x_{per}$ for some \(k \in \bb{Z}_+\), for each $\epsilon \in \bb{Z}_+$, \[\hat{A}_{\infty}^{(\epsilon)} := \{g+t\vec{v}_{\pmb{\ell}_{J-1}} : g \in \hat{A}_{\infty}, \ t \in \bb{Z}_+, \ \dist(g+t\vec{v}_{\pmb{\ell}_{J-1}},\pmb{\ell}_J) \leq d_{\epsilon}\}\] is an \((\pmb{\ell}_J,\pmb{\ell}_{\iota})\)-region and $0 = d_0 < d_1 < \cdots < d_n < \cdots$ is the sequence where, for each $g \in \mathcal{H}(\pmb{\ell}_J)$, there exists $i \in \bb{Z}_+$ such that $\dist(g,\pmb{\ell}_J) = d_i$. 
\end{enumerate} 
\end{lemma}
\begin{proof}
We will prove item (i). The proof of item (ii) is similar. To simplify the notation, we will write \(\pmb{\ell} = \pmb{\ell}_{\iota}\). Let $B' \subset \bb{Z}^2$, with $(0,0) \in B'$, be an $E(\mathcal{S}_{\varphi})$-enve- loped such that the support line of \(B'\) determined by \(\pmb{\ell}\) coincides with $\pmb{\ell}^{(-)}$ and let $A_0 \subset B'$ be a non-empty set. The assumption on item (i) allows us to construct a sequence of $E(\mathcal{S}_{\varphi})$-enveloped sets \[B' \subset B_1 \subset A_{1} \subset B_2 \subset A_2 \subset \cdots \subset B_i \subset A_i \subset \cdots\] such that, for each $i \in \bb{N}$,

\medbreak 
\begin{enumerate}[(i)]\setlength{\itemsep}{5pt}
	\item $B_i \subset \bb{Z}^2$ is an $E(\mathcal{S}_{\varphi})$-enveloped set with $B_i \cap \pmb{\ell}_{ B_i} \subset \pmb{\ell}^{(-)}$,
	\item $B_i$ contains both $A_{i-1}$ and $[-i+1,i-1]^2 \cap \mathcal{H}(\pmb{\ell}^{(-)})$,
	\item $(T^{u_i}\eta)\sob{B_i} = x_{per}\sob{B_i}$ for some $u_i \in \bb{Z}^2$, but $(T^{u_i}\eta)\sob{H_{B_i}(\pmb{\ell})} \neq x_{per}\sob{H_{B_i}(\pmb{\ell})}$,
	\item fixed a sequence $(u_i)_{i \in \bb{N}} \subset \bb{Z}^2$ fulfilling the previous item, $A_i$ is a maximal set with respect to partial ordering by inclusion among all $E(\mathcal{S}_{\varphi})$-enveloped sets $\mathcal{T} \subset \bb{Z}^2$ such that $B_i \subset \mathcal{T} \subset H_{B_i}(\pmb{\ell})$ and $(T^{u_i}\eta)\sob{\mathcal{T}} = x_{per}\sob{\mathcal{T}}$ (see Fi\-gu\-re~\ref{maxsetsAi}).
\end{enumerate}

\begin{figure}[ht]
	\centering\includegraphics[width=7.4cm]{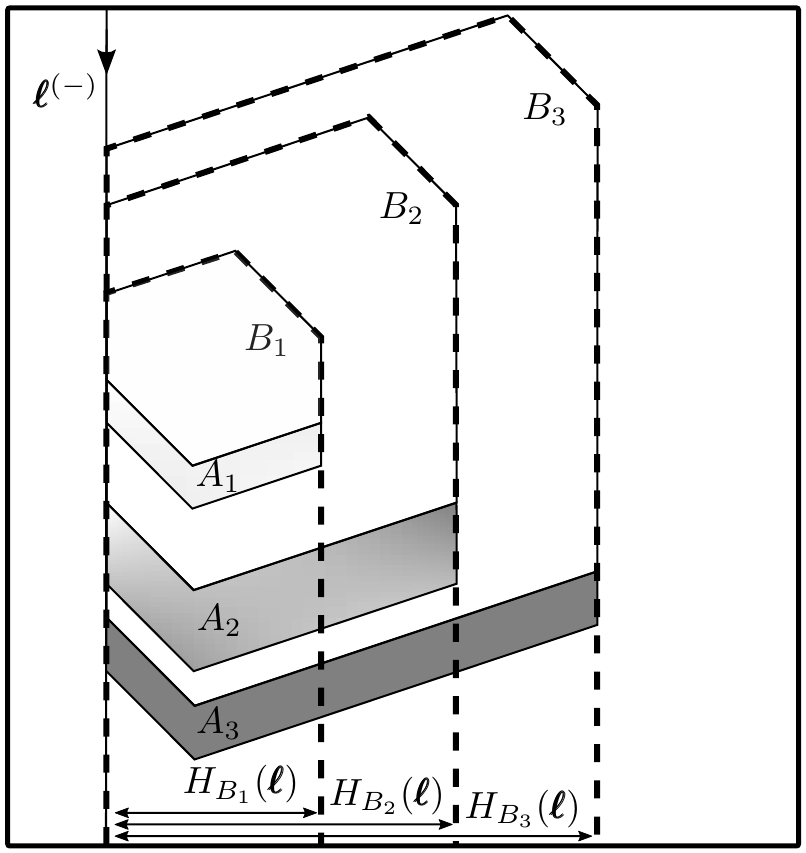}\\
	\caption{The sets $B_1 \subset A_1 \subset B_2 \subset A_2 \subset \cdots$.}
	\label{maxsetsAi}
\end{figure}

Let $g_1 \in A_1$ denote the final point of $A_1 \cap \pmb{\ell}^{(-)}$ with respect to the orientation of $\pmb{\ell}^{(-)}$ and, for each $i > 1$, let $k_i \in \bb{N}$ be such that the final point of $(A_i-k_i\vec{v}_{\pmb{\ell}}) \cap \pmb{\ell}^{(-)}$ with respect to the orientation of $\pmb{\ell}^{(-)}$ coincides with $g_1$. Setting $\hat{A}_i := A_i-k_i\vec{v}_{\pmb{\ell}}$ for all $i \in \bb{N}$, where $k_1 = 0$, then \[(T^{k_i\vec{v}_{\pmb{\ell}}+u_i}\eta)\sob{\hat{A}_i} = (T^{u_i}\eta)\sob{A_i} = x_{per}\sob{A_i} = (T^{k_i\vec{v}_{\pmb{\ell}}}x_{per})\sob{\hat{A}_i}.\] It is easy to see that each $\hat{A}_i$ is a maximal set among all $E(\mathcal{S}_{\varphi})$-enveloped sets $\mathcal{T} \subset \bb{Z}^2$ such that \[B_{i}-k_i\vec{v}_{\pmb{\ell}} \subset \mathcal{T} \subset H_{B_i}(\pmb{\ell})-k_i\vec{v}_{\pmb{\ell}} \quad \text{and} \quad (T^{k_i\vec{v}_{\pmb{\ell}}+u_i}\eta)\sob{\mathcal{T}} = (T^{k_i\vec{v}_{\pmb{\ell}}}x_{per})\sob{\mathcal{T}}.\] Since $x_{per} \in X_{\eta}$ is a periodic configuration with period parallel to $\pmb{\ell}$, then there exists an integer $k \geq 0$ such that $T^{k\vec{v}_{\pmb{\ell}}}x_{per} = T^{k_i\vec{v}_{\pmb{\ell}}}x_{per}$ for infinitely many $i$. By passing to a subsequence, we can as\-sume this holds for all $i$.

Let $w_i(j) \in E(\hat{A}_i)$ denote the edge of $\hat{A}_i$ that is parallel to the oriented line $\pmb{\ell}_j$. Since $\bigcup_{i=1}^{\infty} A_i = \mathcal{H}(\pmb{\ell}^{(-)})$, let $\iota+1 \leq J \leq \iota+m-1$ be the smallest integer such that \[\left|\hat{A}_i \cap w_i(J)\right| < \left|\hat{A}_{i+1} \cap w_{i+1}(J)\right|\] for infinity many $i$. By passing to a subsequence, we can assume that this holds for all $i$. If $J > \iota+1$, we also may assume that \[|\hat{A}_i \cap w_i(j)| = |\hat{A}_{i+1} \cap w_{i+1}(j)|\] for every $\iota+1 \leq j \leq J-1$ and all $i$. In particular, $\hat{A}_{\infty} := \bigcup_{i=1}^{\infty} \hat{A}_i$ is a weakly $E(\mathcal{S}_{\varphi})$-enveloped set (see Figure \ref{maxsetsAihat}), with two semi-infinite edges, one of which is parallel to $\pmb{\ell}$ and the other one is parallel to \(\pmb{\ell}_J\). Actually, \(\hat{A}_{\infty}\) is an \((\pmb{\ell},\pmb{\ell}_J)\)-region. 

\begin{figure}[ht]
	\centering\includegraphics[width=7cm]{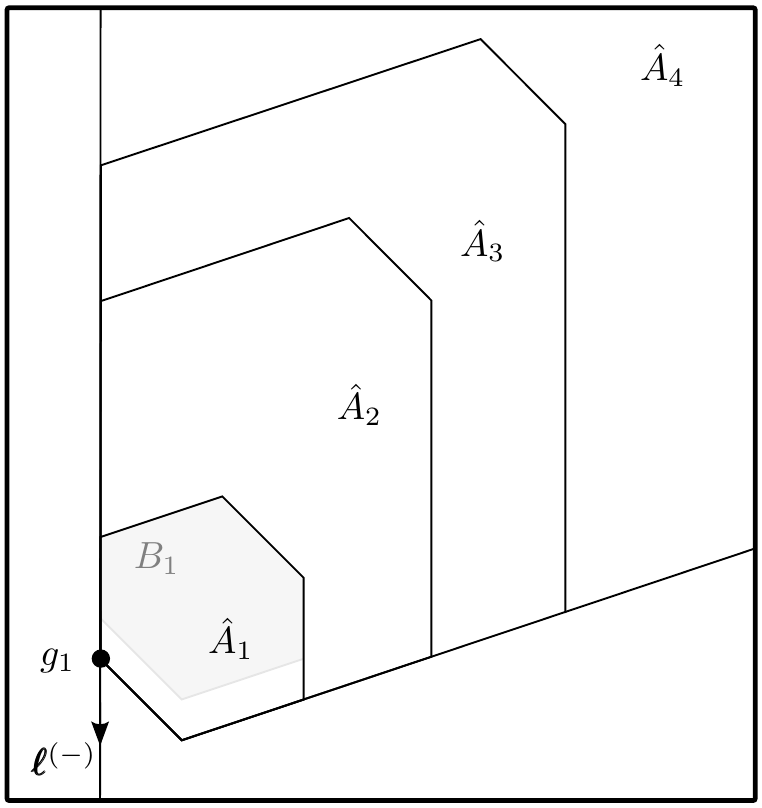}\\
	\caption{The sets $B_1 \subset \hat{A}_1 \subset \hat{A}_2 \subset \cdots \subset \hat{A}_{\infty}$.}
	\label{maxsetsAihat}
\end{figure}

We define $\vartheta_i := T^{k\vec{v}_{\pmb{\ell}}+u_i}\eta$ for all $i$ and $\hat{x}_{per} := T^{k\vec{v}_{\pmb{\ell}}}x_{per}$. Since $\vartheta_{i'}\sob{\hat{A}_i} = \vartheta_i\sob{\hat{A}_i}$ for all $i \leq i'$, by compactness of $X_{\eta}$, the sequence $(\vartheta_i)_{i \in \bb{N}}$ has an accumulation point $\vartheta \in X_{\eta}$ such that $\vartheta\sob{\hat{A}_{\infty}} = \hat{x}_{per}\sob{\hat{A}_{\infty}}$. In particular, one has that  $\vartheta\sob{\hat{A}_{\infty}}$ is a perio\-dic con\-fig\-ur\-ation with period parallel to $\pmb{\ell}$.

We claim that there exists $\epsilon \in \bb{Z}_+$ such that
\begin{equation*}
	\vartheta\sob{\hat{A}_{\infty}^{(\epsilon)}} = \hat{x}_{per}\sob{\hat{A}_{\infty}^{(\epsilon)}}, \ \ \textrm{but} \ \ \vartheta\sob{\hat{A}_{\infty}^{(\epsilon+1)}} \neq \hat{x}_{per}\sob{\hat{A}_{\infty}^{(\epsilon+1)}}.
\end{equation*}
Indeed, suppose, by contradiction, that $\vartheta\sob{\hat{A}_{\infty}^{(\epsilon)}} = \hat{x}_{per}\sob{\hat{A}_{\infty}^{(\epsilon)}}$ for all $\epsilon \in \bb{Z}_+$. Note that, for an appropriate $\epsilon \in \bb{N}$ fixed and all $i \in \bb{N}$ sufficiently large, \[\hat{A}_i^{(\epsilon)} := \{g-t\vec{v}_{\pmb{\ell}_{J+1}} \in \hat{A}^{(\epsilon)}_{\infty} : g \in \hat{A}_i, \ t \in \bb{Z}_+\}\] is an $E(\mathcal{S}_{\varphi})$-enveloped set. Furthermore, one has $B_i-k_i\vec{v}_{\pmb{\ell}} \subset \hat{A}_i^{(\epsilon)} \subset H_{B_i}(\pmb{\ell})-k_i\vec{v}_{\pmb{\ell}}$. So let $\epsilon \in \bb{N}$ and $i_0 \in \bb{N}$ be such that $\hat{A}_{i_0}^{(\epsilon)}$ is an $E(\mathcal{S}_{\varphi})$-en\-ve\-lo\-ped set and consider a constant $I_0 \in \bb{N}$ such that \[\vartheta\sob{\hat{A}_{i_0}^{(\epsilon)}} = \vartheta_i\sob{\hat{A}_{i_0}^{(\epsilon)}} \quad \forall \ i \geq I_0.\] From our assumption it follows that $\vartheta\sob{\hat{A}^{(\epsilon)}_{i}} =$ $\hat{x}_{per}\sob{\hat{A}^{(\epsilon)}_{i}}$ for all $i$. Hence, if we consider $i \geq \max\{i_0,I_0\}$, then $\hat{A}_{i}^{(\epsilon)}$ is an $E(\mathcal{S}_{\varphi})$-en\-ve\-lo\-ped set and $\vartheta_i\sob{\hat{A}^{(\epsilon)}_{i_0}} = \hat{x}_{per}\sob{\hat{A}^{(\epsilon)}_{i_0}}$, which yields
\begin{equation}\label{eq_cont_case2}
	\vartheta_i\sob{\hat{A}_{i} \cup \hat{A}^{(\epsilon)}_{i_0}} = \hat{x}_{per}\sob{\hat{A}_{i} \cup \hat{A}^{(\epsilon)}_{i_0}}.
\end{equation} 
Since $\mathcal{S}_{\varphi}$ is $\eta$-generating, by induction we get from (\ref{eq_cont_case2}) that $\vartheta_i\sob{\hat{A}^{(\epsilon)}_{i}} = \hat{x}_{per}\sob{\hat{A}^{(\epsilon)}_{i}}$, which contradicts the maximality of $\hat{A}_{i}$ and concludes the proof.
\end{proof}

The proof of Theorem \ref{secondary_thm} will be done by contradiction. The main idea is  to reach a contradiction by showing the existence of a line \(\ell' \in \nexpl(\eta)\) such that \(-\pmb{\ell}',\pmb{\ell}' \in \nexpd(\eta)\).  

\subsection{Proof of Theorem \ref{secondary_thm}} 

Suppose, by contradiction, that \(\eta\) is not fully periodic. In particular, as $-\pmb{\ell} \not\in \nexpd(\eta)$ or $\pmb{\ell} \not\in \nexpd(\eta)$ for all lines $\ell \subset \bb{R}^2$ through the origin, Proposition \ref{pps_par_antipar_exp} implies that \(\eta\) is non-periodic. 

According to Theorem \ref{EtienneMainThm}, there exists a configuration \(x_{per} \in X_{\eta}\) where, for every line $\ell \subset \bb{R}^2$ through the origin, $-\pmb{\ell} \not\in \nexpd(x_{per})$ whenever $\pmb{\ell} \not\in \nexpd(x_{per})$. Since $-\pmb{\ell} \not\in \nexpd(x_{per})$ or $\pmb{\ell} \not\in \nexpd(x_{per})$ for all lines $\ell \subset \bb{R}^2$ through the origin, we get that all oriented lines through the origin are one-sided expansive directions on \(\overline{Orb\,(x_{per})}\), which due to the Boyle-Lind Theorem means that \(x_{per}\) is fully periodic.

Let \(\eta = \eta_1+\cdots+\eta_m\) be a \(\bb{Z}\)-minimal periodic decomposition (see Theorem \ref{theorKS}), \(h_i \in \bb{Z}^2\) a period for \(\eta_i\), with \(1 \leq i \leq m\), and \(\varphi(X) := (X^{h_1}-1) \cdots (X^{h_m}-1)\). Suppose $\pmb{\ell}_1,\ldots,\pmb{\ell}_{2m} \subset \bb{R}^2$ is an enumeration of the oriented lines through the origin parallels to the edges of $\mathcal{S}_{\varphi}$ where the edge parallel to $\pmb{\ell}_{i+1}$ is a successor of the edge parallel to $\pmb{\ell}_{i}$ and indices are taken modulo $2m$. Renaming the vectors \(h_1, \ldots, h_m\) if necessary, we may assume that \(h_i\) is either parallel or antiparallel to \(\pmb{\ell}_i\) for every \(1 \leq i \leq m\). 

As \(x_{per}\) is fully periodic, then \(x_{per}\) is periodic with period parallel to any \(\pmb{\ell}_i\), with \(1 \leq i \leq m\).

\begin{claim}\label{claim_coincidebutdiff}
Given an \(E(\mathcal{S}_{\varphi})\)-enveloped set $B \subset \bb{Z}^2$, there exists $u \in \bb{Z}^2$ such that
\begin{equation}\label{eq_main_cor_equalbutdiff}
(T^u\eta)\sob{B} = x_{per}\sob{B}, \ \ \text{but} \ (T^u\eta)\sob{St_B}(\ell_m) \neq x_{per}\sob{St_B}(\ell_m),
\end{equation}
where $$St_B(\ell_m) := H_B(\pmb{\ell}_m) \cup H_B(-\pmb{\ell}_m)$$ is the strip of \(B\) along of \(\ell_m\).
\end{claim}        

Indeed, suppose, by contradiction, that, for any $u \in \bb{Z}^2$, $$(T^u\eta)\sob{B} = x_{per}\sob{B} \implies (T^u\eta)\sob{St_B(\ell_m)} = x_{per}\sob{St_B(\ell_m)}.$$ Since $B$ is a non-empty, finite set, there exists $u \in \bb{Z}^2$ such that $(T^u\eta)\sob{B} = x_{per}\sob{B}$. Changing the alphabet if necessary, let $p \in \bb{N}$ be a prime number such that $\mathcal{A} \subset \bb{Z}_p$ and consider the $\bb{Z}_p$-periodic decomposition $\eta = \overline{\eta}_1+\cdots+\overline{\eta}_{m}$, where $(\overline{\eta}_i)_g := (\eta_i)_g \mod p$ for all $g \in \bb{Z}^2$. Since \[\psi(X) := (X^{h_1}-1) \cdots (X^{h_{m-1}}-1) \in \ann_{\bb{Z}_p}(\eta-\bar{\eta}_m),\] Lemma~\ref{lemma_supportgenerating} states that $\mathcal{S}_{\psi}$ is an $\eta-\overline{\eta}_{m}$-gene\-ra\-ting set. Moreover, as \(h_1, \ldots, h_m \in \bb{Z}^2\) are vectors in pairwise distinct directions, then \(\mathcal{S}_{\psi}\) does not have any edge parallel to \(-\pmb{\ell}_m\) or \(\pmb{\ell}_m\). From $$(T^u\eta)\sob{St_B(\ell_m)} = x_{per}\sob{St_B}(\ell_m)$$ we get that \((T^u\eta)\sob{St_B(\ell_m)}\) and so \((T^{u}(\eta-\overline{\eta}_{m}))\sob{St_B}(\ell_m)\) is periodic with period parallel to \(\pmb{\ell}_m\). Therefore, we conclude by applying the $\eta-\overline{\eta}_{m}$-gene\-ra\-ting set \(\mathcal{S}_{\psi}\) that \(T^{u}(\eta-\overline{\eta}_{m})\) and so \(T^u\eta\) is periodic with period parallel to \(\pmb{\ell}_m\), which contradicts the non-periodicity of \(\eta\) and establishes the claim.\medbreak

According to Claim \ref{claim_coincidebutdiff}, for every \(E(\mathcal{S}_{\varphi})\)-enveloped set $B \subset \bb{Z}^2$ there is $u \in \bb{Z}^2$ such that \((T^u\eta)\sob{B} = x_{per}\sob{B}\), but \[(T^u\eta)\sob{H_B(-\pmb{\ell}_m)} \neq x_{per}\sob{H_B(-\pmb{\ell}_m)} \ \ \text{or} \ \ (T^u\eta)\sob{H_B(\pmb{\ell}_m)} \neq x_{per}\sob{H_B(\pmb{\ell}_m)}.\] We will focus on the case where, for each $E(\mathcal{S}_{\varphi})$-enveloped set $B' \subset \bb{Z}^2$, there exist an $E(\mathcal{S}_{\varphi})$-enveloped set $B \supset B'$, with \((\pmb{\ell}_m)_{B'} = (\pmb{\ell}_m)_{B}\), and $u \in \bb{Z}^2$ such that $$(T^u\eta)\sob{B} = x_{per}\sob{B}, \ \ \text{but} \ \ (T^u\eta)\sob{H_{B}(\pmb{\ell}_m)} \neq x_{per}\sob{H_{B}(\pmb{\ell}_m)}.$$ So due to item (i) of Lemma \ref{tech_lemma}, there exist an \((\pmb{\ell}_m,\pmb{\ell}_J)\)-region \(\hat{A}_{\infty}\), a configuration \(\vartheta \in X_{\eta}\) and \(\epsilon \in \bb{Z}_+\) such that
\begin{equation}\label{eq_existence_tau_lem1_demonst}
	\vartheta\sob{\hat{A}_{\infty}^{(\epsilon)}} = \hat{x}_{per}\sob{\hat{A}_{\infty}^{(\epsilon)}}, \ \ \textrm{but} \ \ \vartheta\sob{\hat{A}_{\infty}^{(\epsilon+1)}} \neq \hat{x}_{per}\sob{\hat{A}_{\infty}^{(\epsilon+1)}},
\end{equation}
where $\hat{x}_{per} := T^{k\vec{v}_{\ell}}x_{per}$ for some \(k \in \bb{Z}_+\). To simplify the notation, we will write \(\pmb{\ell}' = \pmb{\ell}_J\).

\begin{claim}\label{claim_notfullyper}
The set $\{T^{t\vec{v}_{\pmb{\ell}'}}\vartheta : t \in \bb{Z}_+\}$ does not have fully periodic accumulation points. 	
\end{claim}

Indeed, suppose, by contradiction, that $\{T^{t\vec{v}_{\pmb{\ell}'}}\vartheta : t \in \bb{Z}_+\}$ has a fully periodic accumulation point. Then, for each square $Q = [-i,i]^2 \cap \bb{Z}^2$, with $i \in \bb{N}$, we may find an integer $t_0 \in \bb{Z}_+$ such that $(T^{t_0\vec{v}_{\pmb{\ell}'}}\vartheta)\sob{Q} = \vartheta\sob{Q+t_0\vec{v}_{\pmb{\ell}'}}$ is fully periodic and thus periodic with period parallel to \(\pmb{\ell}_m\). Since \(\hat{x}_{per}\) is periodic with period parallel to \(\pmb{\ell}_m\), \(\vartheta\sob{\hat{A}_{\infty}^{(\epsilon)}} = \hat{x}_{per}\sob{\hat{A}_{\infty}^{(\epsilon)}}\) and the set $Q \cap \hat{A}^{(\epsilon)}_{\infty}$ can be taken as large as we want, then \[\vartheta\sob{\hat{A}^{(\epsilon)}_{\infty} \cup (\hat{A}^{(\epsilon+1)}_{\infty} \cap (Q+t_0\vec{v}_{\pmb{\ell}'}))} = \hat{x}_{per}\sob{\hat{A}^{(\epsilon)}_{\infty} \cup (\hat{A}^{(\epsilon+1)}_{\infty} \cap (Q+t_0\vec{v}_{\pmb{\ell}'}))}\] for a square \(Q\) large enough and some \(t_0 \in \bb{N}\). Since $\mathcal{S}_{\varphi}$ is an $\eta$-generating set, we may enlarge the set where \(\vartheta\) and \(\hat{x}_{per}\) coincide so that $\vartheta\sob{\hat{A}^{(\epsilon+1)}_{\infty}} = \hat{x}_{per}\sob{\hat{A}^{(\epsilon+1)}_{\infty}}$,  which contradicts (\ref{eq_existence_tau_lem1_demonst}) and proves the claim.\medbreak

To reach a contradiction and concludes the proof, we will show that $-\pmb{\ell}',\pmb{\ell}' \in \nexpd(\eta)$. Since \(\hat{x}_{per}\) is fully periodic, then \(\vartheta\sob{\hat{A}_{\infty}^{(\epsilon)}} = \hat{x}_{per}\sob{\hat{A}_{\infty}^{(\epsilon)}}\) is periodic with period parallel to $\pmb{\ell}'$. Hence, Pro\-po\-si\-tion~\ref{prop_dec_period_half_plane} implies that any accumulation point $y_{per}$ of $\{T^{t\vec{v}_{\pmb{\ell}'}}\vartheta : t \in \bb{Z}_+\}$ is periodic with period parallel to $\pmb{\ell}'$, but, according to Claim \ref{claim_notfullyper}, not fully periodic. Therefore, from Boyle-Lind Theorem and Proposition~\ref{pps_par_antipar_exp} result that $-\pmb{\ell}',\pmb{\ell}' \in \nexpd(y_{per}) \subset \nexpd(\eta)$, which is a contradiction.

The proof in the case described on item (ii) of Lemma \ref{tech_lemma} is similar to the previous one. \hfill $\Box$

\section{Preliminaries and proof of Theorem \ref{main_thm}}
\label{sec4}

\begin{lemma}\label{lem_subrll'egion}
Let $\eta \in \mathcal{A}^{\bb{Z}^2}$\! be a configuration and suppose there is a line \(\ell \in \nexpl(\eta)\) such that $-\pmb{\ell},\pmb{\ell} \in \nexpd(\eta)$.
\begin{enumerate}[(i)]\setlength{\itemsep}{5pt}
	\item Let $\mathcal{R} \subset \bb{Z}^2$ be an $(\pmb{\ell},\pmb{\ell}')$-region and suppose $\mathcal{S} \subset \bb{Z}^2$ is an $\eta$-generating set, with $\mathcal{S} \backslash \pmb{\ell}'_{\mathcal{S}} \subset \mathcal{R}$, such that \[P_{\eta}(\mathcal{S})-P_{\eta}(\mathcal{S} \backslash \pmb{\ell}'_{\mathcal{S}}) \leq |\mathcal{S} \cap \pmb{\ell}'_{\mathcal{S}}|-1.\] If $x \in X_{\eta}$ is an $(\pmb{\ell}',\mathcal{S},+)$-semi-ambiguous configuration and $x\sob{\mathcal{R}}$ is periodic with period parallel to $\pmb{\ell}$, then there exists an $(\pmb{\ell},\pmb{\ell}')$-region $\mathcal{K} \subset \mathcal{R}$ such that $x\sob{\mathcal{K}}$ is fully periodic with a period parallel to $\pmb{\ell}$ and another one parallel to $\pmb{\ell}'$.
	
	\item Let $\mathcal{R} \subset \bb{Z}^2$ be an $(\pmb{\ell}',\pmb{\ell})$-region and suppose $\mathcal{S} \subset \bb{Z}^2$ is an $\eta$-generating set, with $\mathcal{S} \backslash \pmb{\ell}'_{\mathcal{S}} \subset \mathcal{R}$, such that \[P_{\eta}(\mathcal{S})-P_{\eta}(\mathcal{S} \backslash \pmb{\ell}'_{\mathcal{S}}) \leq |\mathcal{S} \cap \pmb{\ell}'_{\mathcal{S}}|-1.\] If $x \in X_{\eta}$ is an $(\pmb{\ell}',\mathcal{S},-)$-semi-ambiguous configuration and $x\sob{\mathcal{R}}$ is periodic with period parallel to $\pmb{\ell}$, then there exists an $(\pmb{\ell}',\pmb{\ell})$-region $\mathcal{K} \subset \mathcal{R}$ such that $x\sob{\mathcal{K}}$ is fully periodic with a period parallel to $\pmb{\ell}$ and another one parallel to $\pmb{\ell}'$.
\end{enumerate}
\end{lemma}
\begin{proof}
To fix the ideas, we will focus on the situation presented in Figure \ref{llmaxreg}(A). Since the oriented lines $-\pmb{\ell},\pmb{\ell} \in \nexpd(\eta)$, from Lemma \ref{lem_genset_noedge_expas} we get that $\mathcal{S}$ has an edge parallel to $\pmb{\ell}$ and another one parallel to $-\pmb{\ell}$. We will suppose initially that the edge of \(\mathcal{S}\) parallel to \(\pmb{\ell}\) is less or equal than the edge of \(\mathcal{S}\) parallel to \(-\pmb{\ell}\). Let $g'_0, g'_1 \in \mathcal{S}$ be the initial and the final vertices of $\mathcal{S} \cap \pmb{\ell}'_{\mathcal{S}}$ with respect to the orientation of $\pmb{\ell}'$, respectively. Since $\mathcal{S}$ is convex and by assumption $|\mathcal{S} \cap \pmb{\ell}_{\mathcal{S}}| \leq |\mathcal{S} \cap -\pmb{\ell}_{\mathcal{S}}|$, then the convex set $Q \subset \bb{Z}^2$ whose vertices are \[g'_0,g'_0-(|\mathcal{S} \cap \pmb{\ell}_{\mathcal{S}}|-1)\vec{v}_{\pmb{\ell}},g'_1 \ \ \textrm{and} \ \ g'_1-(|\mathcal{S} \cap \pmb{\ell}_{\mathcal{S}}|-1)\vec{v}_{\pmb{\ell}}\] is contained in $\mathcal{S}$ (see Figure \ref{ll'region_RQwandw'}).
\begin{figure}[ht]
	\centering\includegraphics[width=6.4cm]{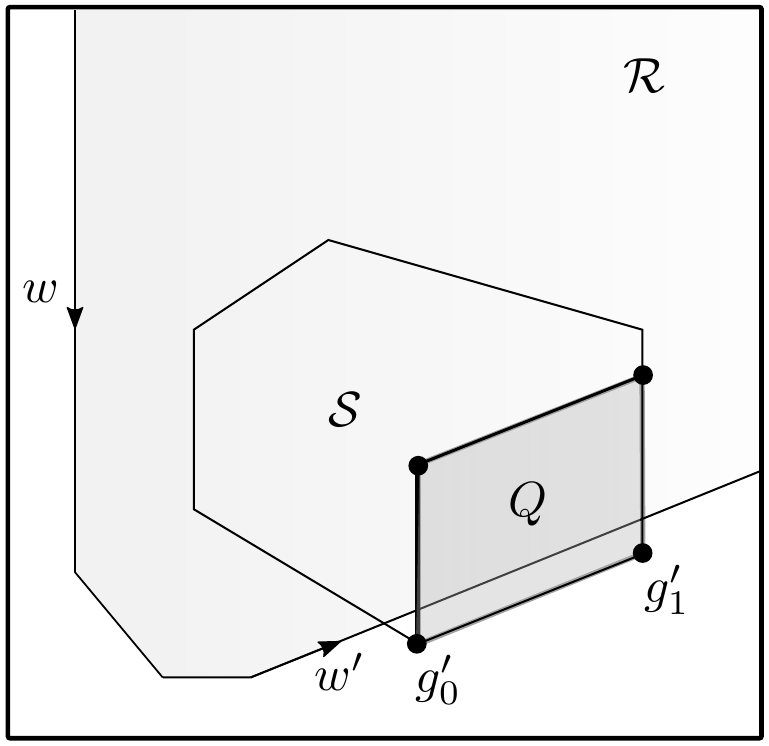}\\
	\caption{The \((\pmb{\ell},\pmb{\ell}')\)-region \(\mathcal{R}\) and the edges \(w \in E(\mathcal{R})\) parallel to \(\pmb{\ell}\) and \(w' \in E(\mathcal{R})\) parallel to \(\pmb{\ell}'\). The sets $\mathcal{S}$ and $Q$ and the points $g'_0$ and $g'_1$.}
	\label{ll'region_RQwandw'}
\end{figure}
Furthermore, for any oriented line $\pmb{\ell}'' \subset \bb{R}^2$ parallel to $\pmb{\ell}'$, with $Q \cap \pmb{\ell}'' \neq \emptyset$, it follows that
\begin{equation}\label{eq_length_edge}
|\mathcal{S} \cap \pmb{\ell}''| \geq |Q \cap \pmb{\ell}''| \geq |Q \cap \pmb{\ell}'_{Q}|-1 = |\mathcal{S} \cap \pmb{\ell}'_{\mathcal{S}}|-1.
\end{equation}

By hypothesis $x$ is an $(\ell',\mathcal{S},+)$-semi-ambiguous configuration. So there is $\tau \in \bb{Z}$, which we may suppose to be positive, satisfying (\ref{semi-ambiguous}).

\begin{claim}
For \(p' := |\mathcal{S} \cap \pmb{\ell}'_{\mathcal{S}}|-1\) and \(H_Q := \left\{g+t\vec{v}_{\pmb{\ell}'} \in \bb{Z}^2 : g \in Q\backslash \pmb{\ell}'_{Q}, \ t \geq \tau+p'\right\}\), $x\sob{H_Q}$ is periodic with period parallel to $\pmb{\ell}'$.
\end{claim}

Indeed, if we define $\pat_{\pmb{\ell}'}\!(\mathcal{S},\tau,x) := \left\{(T^{t\vec{v}_{\pmb{\ell}'}}x)\sob{\mathcal{S} \backslash \pmb{\ell}'_{\mathcal{S}}} : t \geq \tau\right\}$, then \[
\begin{split}
p' \geq P_{\eta}(\mathcal{S})-P_{\eta}(\mathcal{S} \backslash \pmb{\ell}'_{\mathcal{S}}) = \sum_{\gamma \in \pat(\mathcal{S},\eta)} \big(N_{\mathcal{S}}(\pmb{\ell},\gamma)-1\big) & \geq \sum_{\gamma \in \pat_{\pmb{\ell}'}\!(\mathcal{S},\tau,x)} \big(N_{\mathcal{S}}(\pmb{\ell},\gamma)-1\big)\\ & \geq |\pat_{\pmb{\ell}'}\!(\mathcal{S},\tau,x)|.
\end{split}\] 
Given an oriented line $\pmb{\ell}'' \subset \bb{R}^2$ parallel to $\pmb{\ell}'$ satisfying $\pmb{\ell}'' \neq \pmb{\ell}'_Q$ and $Q \cap \pmb{\ell}'' \neq \emptyset$, let\linebreak $u \in Q$ be the initial point of $Q \cap \pmb{\ell}''$ with respect to the orientation of $\pmb{\ell}''$. Since \(Q \subset \mathcal{S}\), from (\ref{eq_length_edge}) we get that \[\mathcal{S} \backslash \pmb{\ell}'_{\mathcal{S}} \supset \left\{u,u+\vec{v}_{\pmb{\ell}'}, \ldots, u+(p'-1)\vec{v}_{\pmb{\ell}'}\right\},\] which yields
\begin{equation}\label{eq_card_period}
p' \geq |\pat_{\pmb{\ell}'}\!(\mathcal{S},\tau,x)| \geq \left|\left\{(T^{t\vec{v}_{\pmb{\ell}'}}x)\sob{\left\{u,u+\vec{v}_{\pmb{\ell}'}, \ldots, u+(p'-1)\vec{v}_{\pmb{\ell}'}\right\}} : t \geq \tau\right\}\right|.
\end{equation}
For $\xi = (\xi_t)_{t \in B} \in \mathcal{A}^{B}$, with \(B = \{\tau,\tau+1,\ldots\}\), defined by $\xi_t = (T^{t\vec{v}_{\pmb{\ell}'}}x)_{u}$ for all \(t \in B\), from (\ref{eq_card_period}) results that $P_{\xi}(p') \leq p'$, or better, that the number of distinct words of length \(p'\) occurring in \(\xi\) is less or equal than \(p'\). Hence, Theorem~\ref{Morse-HedlundThm2} implies that $(\xi_t)_{t \in B+p'}$ is periodic, which means that \(x\sob{\{u+t\vec{v}_{\pmb{\ell}'} : t \geq \tau+p'\}}\) is periodic with period parallel to \(\pmb{\ell}'\). By applying the same reasoning to the others oriented lines \(\pmb{\ell}''\), we obtain that $x\sob{H_Q}$ is periodic with period parallel to $\pmb{\ell}'$, which proves the claim.\medbreak

We remark that \(H_Q\) can be seen as a half strip from \(Q\backslash \pmb{\ell}'_{Q}+(\tau+p')\vec{v}_{\pmb{\ell}'}\) in the direction of \(\pmb{\ell}'\). Furthermore, as \(\tau\) is supposed to be positive, we get that \(H_Q \subset \mathcal{R}\). 

Let $h' \in \bb{Z}^2$ be a period for $x\sob{H_Q}$ parallel to \(\pmb{\ell}'\) and suppose that $\mathcal{T} \subset \mathcal{R}$ is an $E(\mathcal{S})$-enveloped set large enough so that, for any integers $0 \leq r < s$,
\begin{equation}\label{eq_prin_period}
x\sob{\mathcal{T}+r h'} = x\sob{\mathcal{T}+s h'} \ \Longrightarrow \ x\sob{H_{\mathcal{T}}+r h'} = x\sob{H_{\mathcal{T}}+s h'},
\end{equation} 
where \[H_{\mathcal{T}} := \{g-t\vec{v}_{\pmb{\ell}} \in \bb{Z}^2 : g \in \mathcal{T}, \ t \in \bb{Z}_+\}\] is the half-strip from $\mathcal{T}$ in the direction of $-\pmb{\ell}$. Of course this is possible because by assumption $x\sob{\mathcal{R}}$ is periodic with period parallel to $\pmb{\ell}$. Translating $\mathcal{T}$ if necessary, we may suppose that $H_{\mathcal{T}} \cap (H_Q-t\vec{v}_{\pmb{\ell}}) \neq \emptyset$ for all $t \in \bb{N}$ sufficiently large.

To simplify the notation, we will write \(Q' := Q+(\tau+p')\vec{v}_{\pmb{\ell}'}\).

\begin{claim}\label{claim_fullyperiodic}
If $x\sob{\mathcal{T}+r h'} = x\sob{\mathcal{T}+s h'}$ holds for some $0 \leq r < s$, then the restriction of $x$ to the half strip $\mathcal{H}(\pmb{\ell}_{Q'}) \cap \mathcal{R} \cap \mathcal{H}(-\pmb{\ell}_{\mathcal{T}+r h'})$ is periodic of period $(s-r)h'$.
\end{claim}

Indeed, let \(h \in \bb{Z}^2\) be a period for $x\sob{\mathcal{R}}$ parallel to \(\pmb{\ell}\). Since $$x\sob{H_{\mathcal{T}}+rh'} = x\sob{H_{\mathcal{T}}+sh'} = (T^{(s-r)h'}x)\sob{H_{\mathcal{T}}+rh'},$$ then
\begin{equation}\label{eq_ext_period}
x\sob{(H_Q-\iota h) \cup (H_{\mathcal{T}}+r h')} = (T^{(s-r)h'}x)\sob{(H_Q-\iota h) \cup (H_{\mathcal{T}}+r h')} \quad \forall \iota \in \bb{N}.
\end{equation}  
Let $\pmb{\ell}_0 := \pmb{\ell}_{\mathcal{T}+r h'}$ and set $\pmb{\ell}_{i+1} := \pmb{\ell}_{i}^{(-)}$ for all $i \in \bb{Z}_+$. Let $\iota \in \bb{N}$ be large enough so that \((H_Q-i h) \cap (H_{\mathcal{T}}+r h') \neq \emptyset\) for all \(i \geq \iota-|\mathcal{T}|\). Since the convex set \(\pmb{\ell}_1 \cap (H_Q-\iota h)\) has at least \(|\mathcal{S} \cap \pmb{\ell}_{\mathcal{S}}|-1\) elements and $\mathcal{S}$ is an $\eta$-generating set, from (\ref{eq_ext_period}) we can, by induction, enlarge the set where $x$ and $T^{(s-r)h'}x$ coincide by including a subset of $\pmb{\ell}_1 \cap \mathcal{R}$, which can be as large as we want (see Figure \ref{fig8}).
\begin{figure}[ht]
	\centering\includegraphics[width=7.4cm]{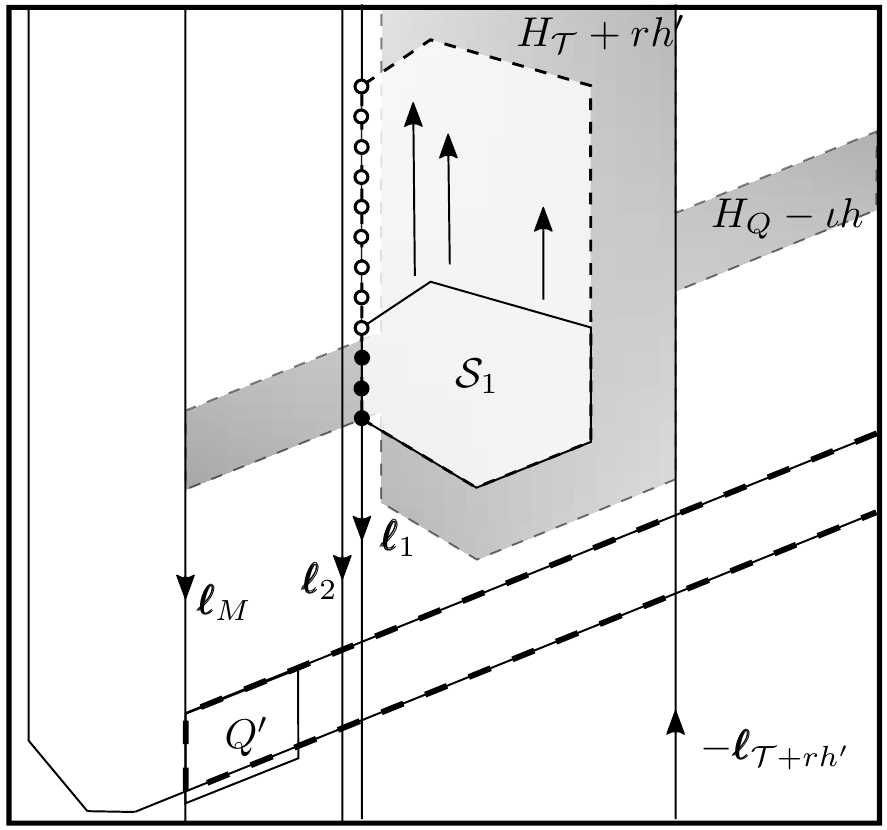}\\
	\caption{The grey region represents the set where the restrictions of $x$ and $T^{(s-r)h'}x$ coincide. The set $\mathcal{S}_1$ denotes the translation of $\mathcal{S}$ where the final points of $\mathcal{S}_1 \cap \pmb{\ell}_1$ and $(H_Q-\iota h) \cap \pmb{\ell}_1$ with respect to the orientation of $\pmb{\ell}_1$ coincide. The white points denote the new set where, after the induction process, the restrictions of $x$ and $T^{(s-r)h'}x$ must coincide.}
	\label{fig8}
\end{figure}
Thus, the periodicity of $x\sob{\pmb{\ell}_1 \cap \mathcal{R}}$ and $(T^{(s-r)h'}x)\sob{\pmb{\ell}_1 \cap \mathcal{R}}$ implies that $x\sob{\pmb{\ell}_1 \cap \mathcal{R}} = (T^{(s-r)h'}x)\sob{\pmb{\ell}_1 \cap \mathcal{R}}$. The same idea applied to the lines $\pmb{\ell}_i$ for $i = 2, \ldots, M$, where $\pmb{\ell}_{M} = \pmb{\ell}_{Q'}$, allows us to conclude that the restriction of $x$ to the half strip $\mathcal{H}(\pmb{\ell}_{Q'}) \cap \mathcal{R} \cap \mathcal{H}(-\pmb{\ell}_{\mathcal{T}+r h'})$ is periodic of period $(s-r)h'$, which proves the claim.\medbreak

Finally, as there exist $0 < t'_0 \leq P_{\eta}(\mathcal{T})+1$ and infinitely many integers $r > 0$ so that (\ref{eq_prin_period}) holds for $r$ and $s = r+t'_0$, Claim \ref{claim_fullyperiodic} implies that $x\sob{\mathcal{H}(\pmb{\ell}_{Q'}) \cap \mathcal{R}}$ is periodic of period $t'_0h'$. 

In the case that $|\mathcal{S} \cap -\pmb{\ell}_{\mathcal{S}}| < |\mathcal{S} \cap \pmb{\ell}_{\mathcal{S}}|$, the vertices of the convex set $Q \subset \bb{Z}^2$ are defined to be \[g'_0,g'_0-(|\mathcal{S} \cap -\pmb{\ell}_{\mathcal{S}}|-1)\vec{v}_{\pmb{\ell}},g'_1 \ \ \textrm{and} \ \ g'_1-(|\mathcal{S} \cap -\pmb{\ell}_{\mathcal{S}}|-1)\vec{v}_{\pmb{\ell}}\] and we consider, instead of the oriented lines $\pmb{\ell}_i$, the oriented lines $-\pmb{\ell}_0 := -\pmb{\ell}_{\mathcal{T}+rh'}$ and $-\pmb{\ell}_{i+1} := -\pmb{\ell}_i^{(-)}$ for all $i \in \bb{Z}_+$. Here, for $0 \leq r < s$ such that $x\sob{\mathcal{T}+r h'} = x\sob{\mathcal{T}+s h'}$, the restriction of $x$ to $\mathcal{H}(\pmb{\ell}_{\mathcal{T}+rh'}) \cap \mathcal{R}$ is periodic of period $(s-r)h'$, which proves (i). 
 
The proof of item (ii) employs similar reasoning. 
\end{proof}

\begin{lemma}\label{lem_no_fullyperiodicaccpoint}
Let $\vartheta \in \mathcal{A}^{\bb{Z}^2}$, with $\mathcal{A} \subset \bb{Z}$, be a non-periodic configuration. Suppose $\vartheta = \vartheta_1 + \cdots + \vartheta_m$ is a $\bb{Z}$-minimal periodic decomposition and $h_i \in \bb{Z}^2$ is a period for $\eta_i$, with \(1 \leq i \leq m\). Then, for each $1 \leq \iota \leq m$, the set $\{T^{th_{\iota}}\vartheta : t \in \bb{Z}\}$ does not have fully periodic accumulation points.
\end{lemma}
\begin{proof}
Given $1 \leq \iota \leq m$, suppose, by con\-tra\-dic\-tion, that $x \in \overline{\{T^{th_{\iota}}\vartheta : t \in \bb{Z}\}}$ is fully periodic and let $(t_n)_{n \in \bb{N}} \subset \bb{Z}$ be a sequence such that $\lim_{n \to +\infty} T^{t_nh_{\iota}}\vartheta = x$. Let $h \in \bb{Z}^2$ be a period for $x$ so that $h$ and $h_{\iota}$ are in distinct direction and consider the Laurent polynomial $\sigma_\iota(X) := \prod_{i \neq \iota} (X^{h_i}-1)$. Note that, for any box $B \subset \bb{Z}^2$ as large as we want, we can find an index $n \in \bb{N}$ such that $(\sigma_{\iota}(X)\vartheta)\sob{B+t_nh_{\iota}}$ is periodic of period $h$. 

We claim that $\sigma_{\iota}(X)\vartheta$ is periodic of period $h$. Indeed, given $g \in \bb{Z}^2$, let $B \subset \bb{Z}^2$ be a box large enough so that, for some $t \in \bb{Z}$, \[g+th_{\iota},g+th_{\iota}+h  \in B,\] and let $n \in \bb{N}$ be such that $(\sigma_{\iota}(X)\vartheta)\sob{B+t_nh_{\iota}}$ is periodic of period $h$. By using that $\sigma_{\iota}(X)\vartheta =$ $\sigma_{\iota}(X)\vartheta_{\iota}$ is periodic of period $h_{\iota}$, it follows that \[(\sigma_{\iota}(X)\vartheta)_g = (\sigma_{\iota}(X)\vartheta)_{g+(t+t_n)h_{\iota}} = (\sigma_{\iota}(X)\vartheta)_{g+(t+t_n)h_{\iota}+h} = (\sigma_{\iota}(X)\vartheta)_{g+h},\] which concludes the claim. 

Set $\varphi(X) := (X^{h_1}-1) \cdots (X^{h_m}-1)$ and $\psi(X) := (X^{h}-1)\sigma_{\iota}(X) \in \ann_{\bb{Z}}(\vartheta)$. Note that \(\mathcal{S}_{\varphi}\) has an edge parallel to \(h_{\iota}\), but $\mathcal{S}_{\psi}$ does not have any edge parallel to $h_{\iota}$, which due to Pro\-po\-si\-tion~\ref{prop_geom_convexset} is a contradiction.
\end{proof}

\subsection{The main steps for the proof of Theorem \ref{main_thm}}
\label{subsec}

The proof of Theorem~\ref{main_thm} will be separated in cases and the main steps are described below: let $\eta \in \mathcal{A}^{\bb{Z}^2}$, with $\mathcal{A} \subset \bb{Z}$, be a non-periodic low convex complexity configuration. Suppose \[\varphi(X) := (X^{h_1}-1) \cdots (X^{h_m}-1) \in \ann_{\bb{Z}}(\eta)\] and let $w_1,\ldots,w_{2m} \in E(\mathcal{S}_{\varphi})$ be an enumeration where \(w_{i+1}\) is a successor edge of \(w_i\) and indices are taken modulo $2m$. Suppose that all non-periodic configurations in $X_{\eta}$ have the same order.

\medbreak\noindent{\it Main lemma of the proof of Theorem \ref{main_thm}.}\hspace{1ex}\ignorespaces
If the antiparallel oriented lines through the origin parallels to \(w_i\) and \(w_{i+m}\) are both one-sided nonexpansive directions on \(X_{\eta}\), then the antiparallel oriented lines through the origin parallels to \(w_{i+m-1}\) and \(w_{i+2m-1}\) are both one-sided nonexpansive directions on \(X_{\eta}\) (See Figure \ref{figmainidea}(A) and Figure~\ref{figmainidea}(B)).\medbreak

To conclude the proof of Theorem \ref{main_thm}, we just need to apply the main lemma successively, starting from a line $\ell \in \nexpl(\eta)$ such that $-\pmb{\ell},\pmb{\ell} \in \nexpd(\eta)$ (see Figu\-re~\ref{figmainidea}). Of course, for a not fully periodic configuration with a non-trivial annihilator such a line always exists (see Theorem \ref{secondary_thm}).

\begin{figure}[!htbp]
	\centering
	\begin{minipage}[b]{0.49\linewidth}
		\centering\includegraphics[width=5cm]{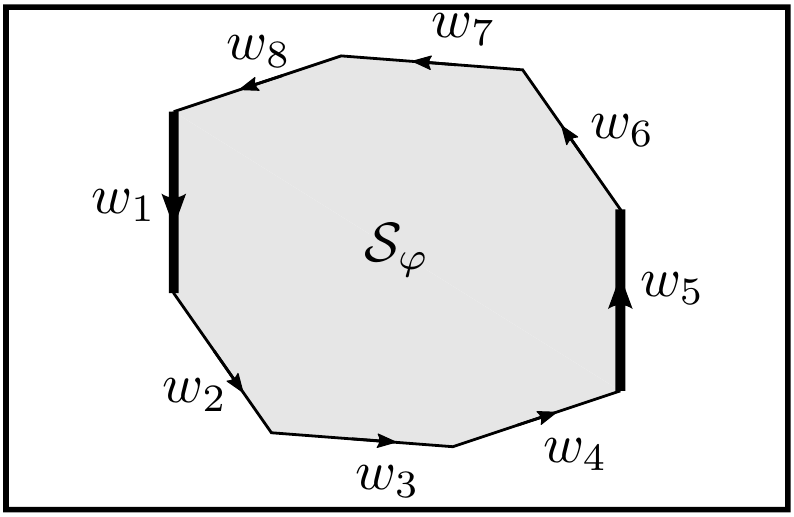}\\
		\small{(A) Suppose that $\pmb{\ell}_1,\pmb{\ell}_5, \in \nexpd(\eta)$.}
	\end{minipage} 
	\hfill
	\begin{minipage}[b]{0.49\linewidth}
		\centering\includegraphics[width=5cm]{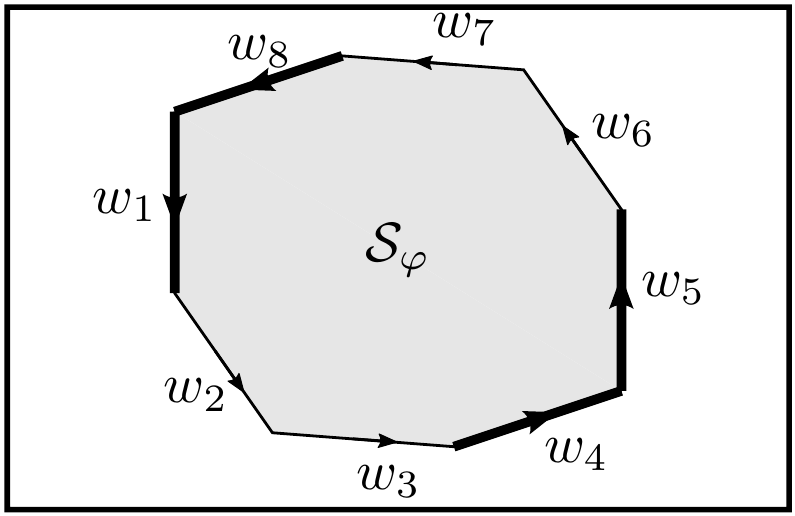}\\
		\small{(B) From (A) we get $\pmb{\ell}_4,\pmb{\ell}_8 \in \nexpd(\eta)$.}
	\end{minipage}
	\hfill
	\vspace{0.4cm}
	\begin{minipage}[b]{0.49\linewidth}
		\centering\includegraphics[width=5cm]{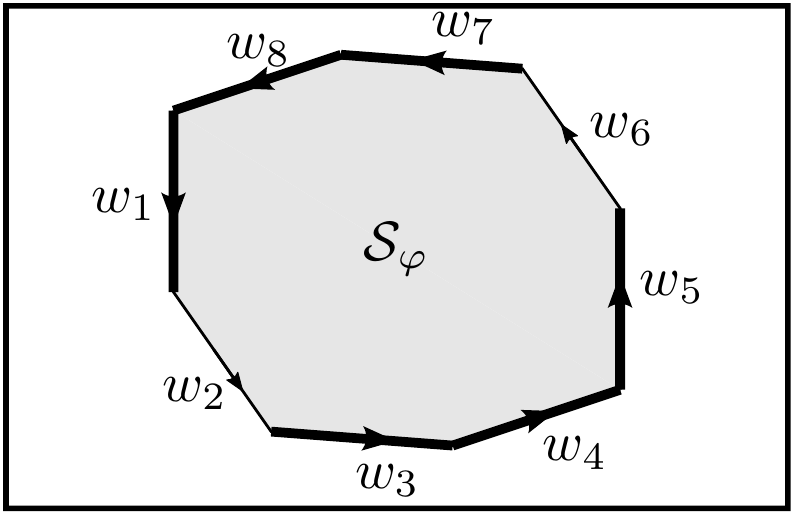}\\
		\small{(C) From (B) we get $\pmb{\ell}_3,\pmb{\ell}_7 \in \nexpd(\eta)$.}
	\end{minipage}
	\hfill
	\begin{minipage}[b]{0.49\linewidth}
		\centering\includegraphics[width=5cm]{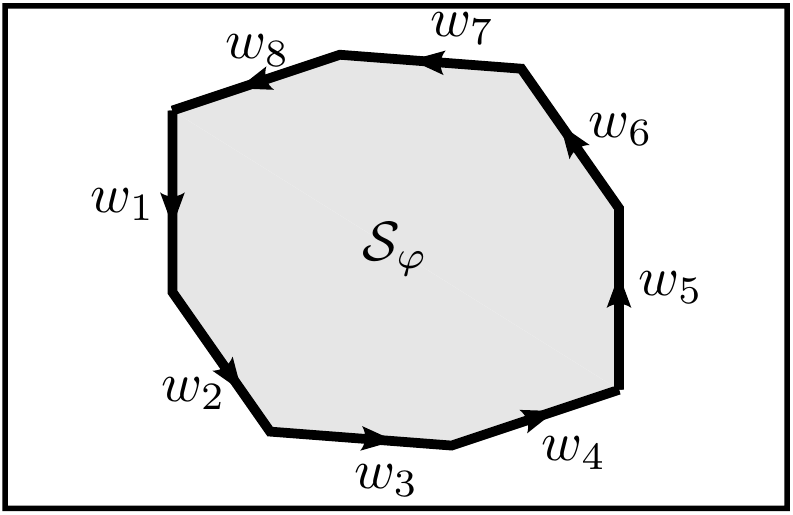}\\
		\small{(D) From (C) we get $\pmb{\ell}_2,\pmb{\ell}_6 \in \nexpd(\eta)$.}
	\end{minipage}
	\caption{Let $\pmb{\ell}_1,\ldots,\pmb{\ell}_{8} \subset \bb{R}^2$ be an enumeration of the oriented lines through the origin where \(\pmb{\ell}_i\) is parallel to \(w_i\) for each \(1 \leq i \leq 8\). We apply the main lemma stated above successively to show that $\pmb{\ell}_1,\ldots,\pmb{\ell}_{8} \in \nexpd(\eta)$. Note that \(\pmb{\ell}_5 = -\pmb{\ell}_1\), \(\pmb{\ell}_8 = -\pmb{\ell}_4\), \(\pmb{\ell}_7 = -\pmb{\ell}_3\) and \(\pmb{\ell}_6 = -\pmb{\ell}_2\).}
	\label{figmainidea}
\end{figure}

Let $\pmb{\ell}_1,\ldots,\pmb{\ell}_{2m} \subset \bb{R}^2$ be an enumeration of the oriented lines through the origin where \(w_i\) is parallel to \(\pmb{\ell}_i\), with \(1 \leq i \leq 2m\). Suppose that \(-\pmb{\ell}_{\iota},\pmb{\ell}_{\iota} \in \nexpd(\eta)\) and let \(x_{per} \in X_{\eta}\) be a periodic configuration with period parallel to \(\pmb{\ell}_{\iota}\). To prove the main lemma, we need to con\-si\-der two cases:

{\it Case 1.} For some \(E(\mathcal{S}_{\varphi})\)-enveloped set \(B \subset \bb{Z}^2\) there exists $u \in \bb{Z}^2$ such that $(T^u\eta)\sob{H_B(\pmb{\ell}_{\iota})} = x_{per}\sob{H_B(\pmb{\ell}_{\iota})}$.

\medbreak

{\it Case 2.} For any $E(\mathcal{S}_{\varphi})$-enveloped set $B \subset \bb{Z}^2$ and all $u \in \bb{Z}^2$ such that $(T^u\eta)\sob{B} = x_{per}\sob{B}$, one has $(T^u\eta)\sob{H_B(\pmb{\ell}_{\iota})} \neq x_{per}\sob{H_B(\pmb{\ell}_{\iota})}$.

\medbreak

In Case 1, we use Lemma \ref{lem_subrll'egion} to construct an \((-\pmb{\ell}_{\iota}, \pmb{\ell}')\)-region whose restriction of \(T^u\eta\) is fully periodic. A certain kind of maximality implies that \(-\pmb{\ell}',\pmb{\ell}' \in \nexpd(\eta)\). To finish, we apply Lemma \ref{lem_no_fullyperiodicaccpoint} to show that the boundary of such an \((-\pmb{\ell}_{\iota}, \pmb{\ell}')\)-region has a rigid structure, what implies that \(\pmb{\ell}'\) is parallel to \(w_{\iota+2m-1} = w_{\iota-1}\). In Case~2 the strategy is similar, but we use Lemma \ref{tech_lemma} to get an \((\pmb{\ell}_{\iota}, \pmb{\ell}')\)-region whose restriction of some non-periodic configuration is fully periodic. Here, the oriented line \(\pmb{\ell}'\) will be parallel to \(w_{\iota+m-1}\).  

\subsection{Proof of Theorem \ref{main_thm}}

If \(\eta\) is periodic, the result follows from Proposition~\ref{pps_par_antipar_exp}. So we may assume \(\eta\) is non-periodic. Let \(\eta = \eta_1+\cdots+\eta_m\) be a \(\bb{Z}\)-minimal periodic decomposition (see Theorem \ref{theorKS}), \(h_i \in \bb{Z}^2\) a period for \(\eta_i\), with \(1 \leq i \leq m\), and set $\varphi(X) := (X^{h_1}-1) \cdots (X^{h_m}-1)$. Suppose $\pmb{\ell}_1,\ldots,\pmb{\ell}_{2m} \subset \bb{R}^2$ is an enumeration of the oriented lines through the origin parallels to the edges of $\mathcal{S}_{\varphi}$ where the edge parallel to $\pmb{\ell}_{i+1}$ is a successor of the edge parallel to $\pmb{\ell}_{i}$ and indices are taken modulo $2m$. Renaming the vectors \(h_1, \ldots, h_m\) if necessary, we may assume that each \(h_i\), with \(1 \leq i \leq m\), is either parallel or antiparallel to \(\pmb{\ell}_i\).

The next lemma is the heart of the proof of Theorem \ref{main_thm} and its proof will be separated in two cases:

\begin{lemma}\label{mainlemmamainthm}
If $-\pmb{\ell}_{\iota},\pmb{\ell}_{\iota} \in \nexpd(\eta)$ for some \(1 \leq \iota \leq 2m\), then $-\pmb{\ell}_{\iota+m-1},\pmb{\ell}_{\iota+m-1} \in \nexpd(\eta)$.
\end{lemma}
\begin{proof} 
To simplify the nota\-tion, we will write $-\pmb{\ell}_{\iota} = -\pmb{\ell}$ and $\pmb{\ell}_{\iota} = \pmb{\ell}$. Let $\mathcal{S} \subset \bb{Z}^2$, with $\mathcal{S} \cap \pmb{\ell}_{\mathcal{S}}$ contained in $\pmb{\ell}^{(-)}$, be an $\eta$-generating set as in Lemma~\ref{lemma_generatingset}, which according to Lemma~\ref{lem_genset_noedge_expas}  has an edge parallel to $\pmb{\ell}$ and another one parallel to $-\pmb{\ell}$. We will suppose initially that \(|\mathcal{S} \cap \pmb{\ell}_{\mathcal{S}}| \leq |\mathcal{S} \cap -\pmb{\ell}_{\mathcal{S}}|\). Since \(\pmb{\ell} \in \nexpd(\eta)\), then there exist configurations $x_{per},y_{per} \in X_{\eta}$ such that $x_{per}\sob{\mathcal{H}(\pmb{\ell})} = y_{per}\sob{\mathcal{H}(\pmb{\ell})}$, but \((x_{per})_g \neq (y_{per})_g\) for some $g \in \pmb{\ell}^{(-)} \cap \bb{Z}^2$. Note that $x_{per}$ and $y_{per}$ are $(\pmb{\ell},\mathcal{S})$-am\-bi\-gu\-ous configurations. Thus, Propositions \ref{prop_CyrKra} and \ref{prop_dec_period_half_plane} imply that $x_{per}$ and $y_{per}$ are periodic with periods parallels to \(\pmb{\ell}\).

\medbreak
\subsection*{Case 1} Let \(B \subset \bb{Z}^2\) be an \(E(\mathcal{S}_{\varphi})\)-enveloped set and suppose there exists $u \in \bb{Z}^2$ such that $(T^u\eta)\sob{H_B(\pmb{\ell})} = x_{per}\sob{H_B(\pmb{\ell})}$.
\medbreak

\begin{claim}
$(T^u\eta)\sob{\mathcal{R}_{\iota-1}}$ is periodic with period parallel to $\pmb{\ell}$, where \[\mathcal{R}_{\iota-1} := \{g+t\vec{v}_{\pmb{\ell}_{\iota-1}} \in \bb{Z}^2 : g \in H_B(\pmb{\ell}), \ t \in \bb{Z}_+\}\] is an $(-\pmb{\ell},\pmb{\ell}_{\iota-1})$-region.
\end{claim} 

Indeed, changing the al\-pha\-bet if necessary, let $p \in \bb{N}$ be a prime number such that $\mathcal{A} \subset \bb{Z}_p$ and consider the $\bb{Z}_p$-periodic decomposition $\eta = \overline{\eta}_1+\cdots+\overline{\eta}_m$, where $(\overline{\eta}_i)_g := (\eta_i)_g \mod p$ for all $g \in \bb{Z}^2$. Let \(1 \leq \iota_0 \leq m\) be such that \(\iota_0 = \iota \mod m\). Note that \(\bar{\eta}_{\iota_0}\) and then $(T^u(\eta-\overline{\eta}_{\iota_0}))\sob{H_B(\pmb{\ell})}$ is periodic with period parallel to \(\pmb{\ell}\). Since \[\varphi_{\iota}(X) := \prod_{i \neq \iota_0} (X^{h_i}-1) \in \ann_{\bb{Z}_p}(\eta-\overline{\eta}_{\iota_0}),\] Lem\-ma~\ref{lemma_supportgenerating} states that $\mathcal{S}_{\varphi_{\iota}}$ is an $\eta-\overline{\eta}_{\iota_0}$-ge\-nerating set. Being $h_1,\ldots,h_m \in \bb{Z}^2$ vectors in pairwise distinct directions, Lemma \ref{lem_periods_annihilator} implies that $\mathcal{S}_{\varphi_{\iota}}$ does not have any edge parallel to $-\pmb{\ell}$ or $\pmb{\ell}$. So we may extend the periodicity from $(T^u(\eta-\overline{\eta}_{\iota_0}))\sob{H_B(\pmb{\ell})}$ to $$(T^u(\eta-\overline{\eta}_{\iota_0}))\sob{H_B(\pmb{\ell}) \cup A_1},$$ where $A_1 := \mathcal{R}_{\iota-1} \cap \pmb{l}_1$ and \(\pmb{l}_1 := \pmb{\ell}^{(-)}_{B}\) (see Figure~\ref{fig10}).     
\begin{figure}[ht]
	\centering\includegraphics[width=7.4cm]{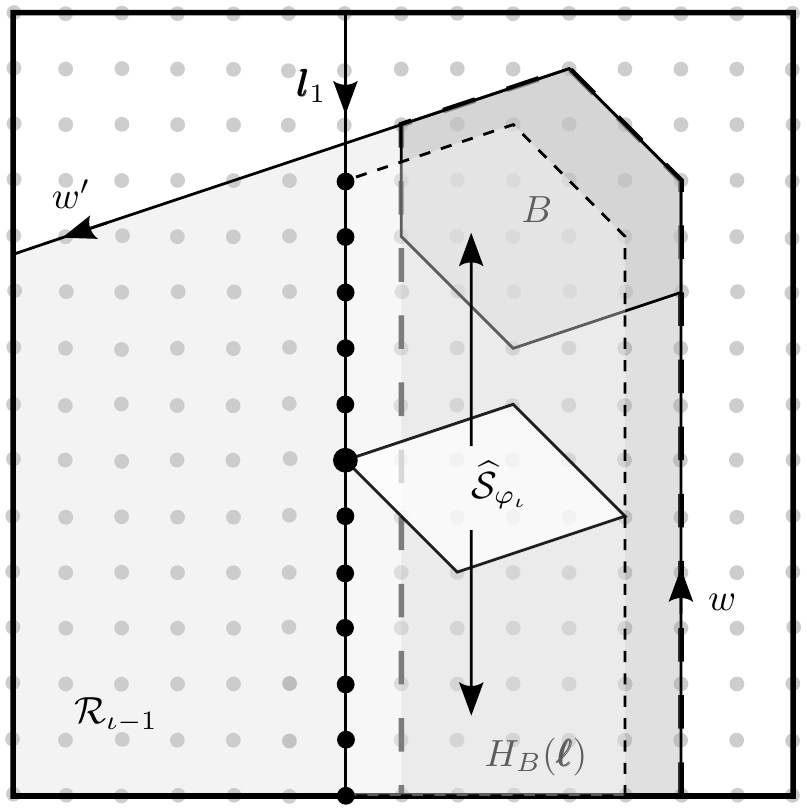}\\
	\caption{The set \(B\), the half strip \(H_B(\pmb{\ell})\), the \((-\pmb{\ell},\pmb{\ell}_{\iota-1})\)-re\-gion \(\mathcal{R}_{\iota-1}\) and the edges \(w \in E(\mathcal{R}_{\iota-1})\) parallel to \(-\pmb{\ell}\) and \(w' \in E(\mathcal{R}_{\iota-1})\) parallel to \(\pmb{\ell}_{\iota-1}\). The set $\widehat{\mathcal{S}}_{\varphi_{\iota}}$ denotes a trans\-lation of $\mathcal{S}_{\varphi_{\iota}}$ and the black points represent the set \(A_1\). Since $\widehat{\mathcal{S}}_{\varphi_{\iota}}$ is an $\eta-\overline{\eta}_{\iota_0}$-ge\-ne\-ra\-ting set, then the knowledge of \(T^u\eta\) on \(H_B(\pmb{\ell})\) determines uniquely \(T^u\eta\) on \(A_1\).}
	\label{fig10}
\end{figure}
As before, we may extend the periodicity from $(T^u(\eta-\overline{\eta}_{\iota_0}))\sob{H_B(\pmb{\ell}) \cup A_1}$ to $$(T^u(\eta-\overline{\eta}_{\iota_0}))\sob{H_B(\pmb{\ell}) \cup A_1 \cup A_2},$$ where $A_2 := \mathcal{R}_{\iota-1} \cap \pmb{l}_2$ and \(\pmb{l}_2 := \pmb{l}^{(-)}_{1}\). Proceeding this way, we get by induction that $(T^u(\eta-\overline{\eta}_{\iota_0}))\sob{\mathcal{R}_{\iota-1}}$ and therefore $(T^u\eta)\sob{\mathcal{R}_{\iota-1}}$ is periodic with period parallel to \(\pmb{\ell}\), which establishes the claim.\medbreak

Let $h \in \bb{Z}^2$ be a period for $(T^u\eta)\sob{\mathcal{R}_{\iota-1}}$ parallel to \(\pmb{\ell}\). For each integer $\iota-m+1 \leq i \leq \iota-2$, we define the $(-\pmb{\ell},\pmb{\ell}_i)$-region \[\mathcal{R}_{i} := \{g+t\vec{v}_{\pmb{\ell}_{i}} \in \bb{Z}^2 : g \in \mathcal{R}_{i+1}, \ t \in \bb{Z}_+\}.\] Note that $(T^u\eta)\sob{\mathcal{H}(\pmb{\ell}_{\iota-m})}$ does not have a period parallel to $\pmb{\ell}_{\iota-m}$ or \(\pmb{\ell} = \pmb{\ell}_{\iota}\), since otherwise Proposition~\ref{prop_dec_period_half_plane} would implies that $T^u\eta$ is periodic. Hence, as \[\mathcal{R}_{\iota-1} \subset \mathcal{R}_{\iota-2} \subset \cdots \subset \mathcal{R}_{\iota-m+1},\] we may consider the smallest integer $\iota-m+1 \leq I \leq \iota-1$ such that $(T^u\eta)\sob{\mathcal{R}_I}$ is periodic of period $h$. 

Let $0 = d_0 < d_1 < \cdots < d_n < \cdots$ denote the sequence where, for each $g \in \mathcal{H}(\pmb{\ell}_I)$, there exists $i \in \bb{Z}_+$ such that $\dist(g,\pmb{\ell}_I) = d_i$. To simplify the notation, we will write $\pmb{\ell}_I = \pmb{\ell}'$. Given $n \in \bb{Z}_+$, we define the $(-\pmb{\ell},\pmb{\ell}')$-region \[\mathcal{R}^{n}_{I} := \{g+t\vec{v}_{\pmb{\ell}_{I-1}} \in \bb{Z}^2 : g \in \mathcal{R}_{I}, \ t \in \bb{Z}_+, \ \dist(g+t\vec{v}_{\pmb{\ell}_{I-1}},\pmb{\ell}'_{\mathcal{R}_I}) \leq d_n\}.\] From the fact that $\bigcup_{n = 0}^{\infty} \mathcal{R}^{n}_{I} = \mathcal{R}_{I-1}$, we may consider the greatest integer $N \in \bb{Z}_+$ such that $(T^u\eta)\sob{\mathcal{R}^{N}_{I}}$ is periodic of period $h$.

\begin{claim}\label{semiambiguous_case1}
\(T^u\eta\) is an \((\pmb{\ell}',\mathcal{T},+)\)-semi-ambiguous configuration for any translation $\mathcal{T}$ of $\mathcal{S}$ where $\mathcal{T} \backslash \pmb{\ell}'_{\mathcal{T}} \subset \mathcal{R}^{N}_{I}$, but $\mathcal{T} \not\subset \mathcal{R}^N_{I}$.
\end{claim}

Indeed, suppose, by contradiction, that there exists an integer $t_0 \in \bb{Z}_+$ such that $N_{\mathcal{T}}(\pmb{\ell}',\gamma) = 1$, where $\gamma =  (T^{t_0\vec{v}_{\pmb{\ell}'}}(T^u\eta))\sob{\mathcal{T} \backslash \pmb{\ell}'_{\mathcal{T}}}$. Being $(T^u\eta)\sob{\mathcal{R}^{N}_{I}}$ periodic of period $h$, we have that \[(T^{t_0\vec{v}_{\pmb{\ell}'}}(T^u\eta))\sob{\mathcal{T} \backslash \pmb{\ell}'_{\mathcal{T}}} = \gamma =  (T^{t_0\vec{v}_{\pmb{\ell}'}+h}(T^u\eta))\sob{\mathcal{T} \backslash \pmb{\ell}'_{\mathcal{T}}}.\] Since $N_{\mathcal{T}}(\pmb{\ell}',\gamma) = 1$, then $$(T^{t_0\vec{v}_{\pmb{\ell}'}}(T^u\eta))\sob{\mathcal{T}} = (T^{t_0\vec{v}_{\pmb{\ell}'}+h}(T^u\eta))\sob{\mathcal{T}}$$ and so
\begin{equation}\label{eq_ext_edge_uni_set}
(T^u\eta)\sob{\mathcal{R}^{N}_I \cup (\mathcal{T}+t_0\vec{v}_{\pmb{\ell}'})} = (T^{h}(T^u\eta))\sob{\mathcal{R}^N_I \cup (\mathcal{T}+t_0\vec{v}_{\pmb{\ell}'})}.
\end{equation}
Since $\mathcal{T}$ is $\eta$-generating, (\ref{eq_ext_edge_uni_set}) holds for all $t_0 \in \bb{Z}_+$ if $\mathcal{T}$ does not have an edge par\-al\-lel to $\pmb{\ell}'$, which means that \[(T^u\eta)\sob{\mathcal{R}^N_I \cup \{g+t\vec{v}_{\pmb{\ell}'} : t \in \bb{Z}_+\}} = (T^{h}(T^u\eta))\sob{\mathcal{R}^N_I \cup \{g+t\vec{v}_{\pmb{\ell}'} : t \in \bb{Z}_+\}},\] where $g \in \mathcal{T} \cap \pmb{\ell}'_{\mathcal{T}}$. Otherwise, (\ref{eq_ext_edge_uni_set}) and the fact that $\mathcal{T}$ is $\eta$-generating allow us to obtain by induction that \[(T^u\eta)\sob{\mathcal{R}^N_I \cup \{g+t\vec{v}_{\pmb{\ell}'} : t \geq t_0\}} = (T^{h}(T^u\eta))\sob{\mathcal{R}^N_I \cup \{g+t\vec{v}_{\pmb{\ell}'} : t \geq t_0\}}\] (see Figure \ref{fig11}(A)). In any case, by using that $\mathcal{S}_{\varphi}$ is $\eta$-generating, we get that $$(T^u\eta)\sob{\mathcal{R}^{N+1}_I} = (T^h(T^u\eta))\sob{\mathcal{R}^{N+1}_I}$$ (see Figure \ref{fig11}(B)), which con\-tra\-dicts the maximality of $N \in \bb{Z}_+$ and proves the claim.\medbreak

\begin{figure}[!htbp]
	\begin{minipage}[t]{0.47\linewidth}
		\begin{center}
			\includegraphics[width=5.4cm]{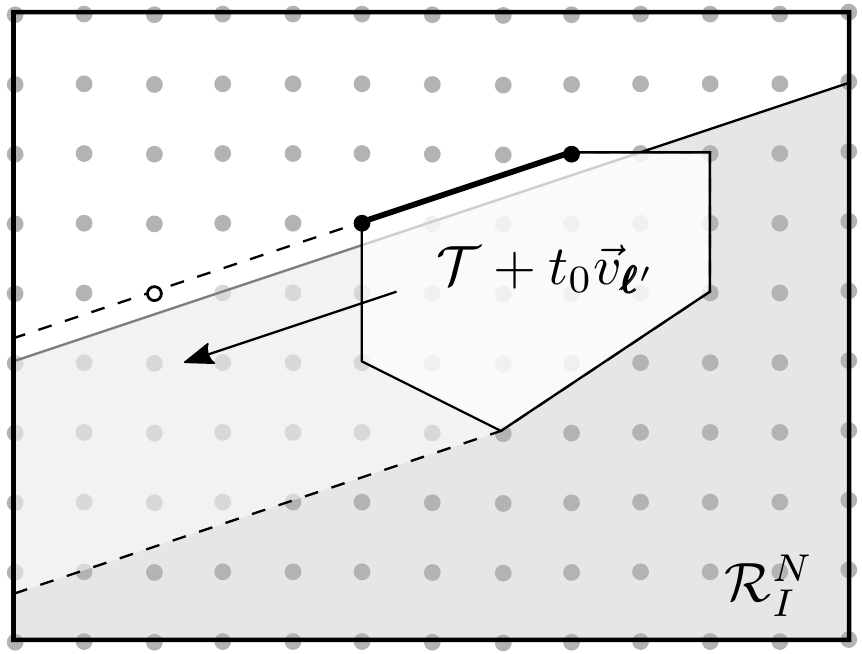}\\
		\end{center}
		\small{(A) Since $\mathcal{T}$ is $\eta$-generating, the knowledge of a configuration on \(\mathcal{R}^N_I \cup (\mathcal{T}+t_0\vec{v}_{\pmb{\ell}'})\) determines unique\-ly such a configuration on \(\{g+t\vec{v}_{\pmb{\ell}'} : t \geq t_0\}\).}
	\end{minipage} 
	\hfill
	\begin{minipage}[t]{0.47\linewidth}
		\begin{center}
			\includegraphics[width=5.4cm]{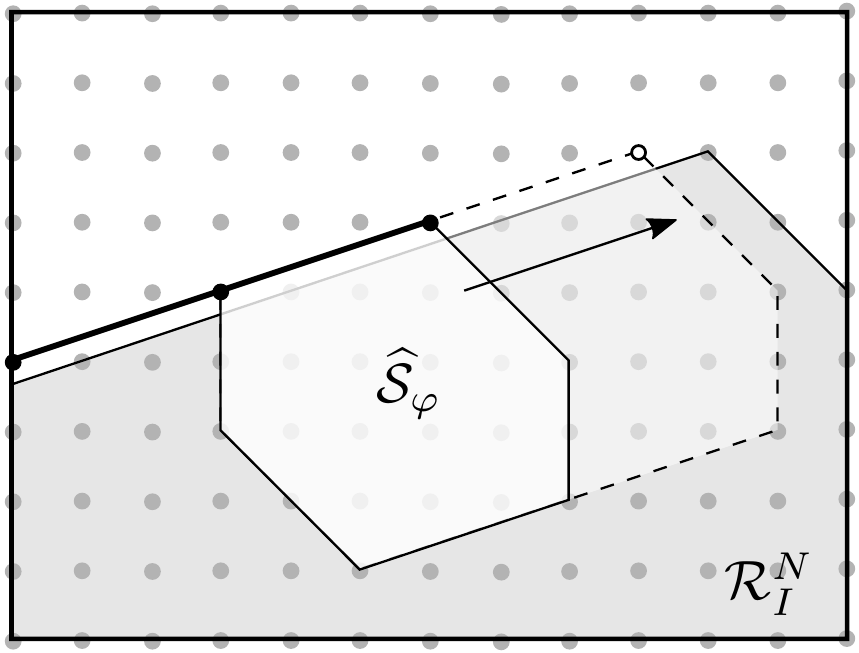}\\
		\end{center}
		\small{(B) Let \(\widehat{\mathcal{S}}_{\varphi}\) denote a translation of \(\mathcal{S}_{\varphi}\). Since $\widehat{\mathcal{S}}_{\varphi}$ is $\eta$-generating, the knowledge of a configuration on $\mathcal{R}^N_I \cup \{g+t\vec{v}_{\pmb{\ell}'} : t \geq t_0\}$ determines uniquely such a configuration on $\mathcal{R}^{N+1}_I$.}
	\end{minipage}
	\caption{The main steps to contradict the maximality of \(N\).}
	\label{fig11}
\end{figure}

Due to Claim~\ref{semiambiguous_case1} and Lemma~\ref{lem_subrll'egion}, there exists an $(-\pmb{\ell},\pmb{\ell}')$-region $\mathcal{K} \subset \mathcal{R}^N_I$ such that $(T^u\eta)\sob{\mathcal{K}}$ is fully periodic with a period parallel to \(\pmb{\ell}\) and another one parallel to $\pmb{\ell}'$.

\begin{claim}\label{notfullyperiodic_case1}
The set $\{T^{t\vec{v}_{\pmb{\ell}'}}(T^u\eta) : t \in \bb{Z}_+\}$ does not have fully periodic accumulation points.
\end{claim}

Indeed, since \(T^u\eta\) is a non-periodic configuration, \(T^u\eta = T^u\eta_1+\cdots+T^u\eta_m\) is a \(\bb{Z}\)-minimal periodic decomposition and \(h_i\) is a period for \(T^u\eta_i\), with \(1 \leq i \leq m\), Lemma~\ref{lem_no_fullyperiodicaccpoint} prevents $\{T^{t\vec{v}_{\pmb{\ell}'}}(T^u\eta) : t \in \bb{Z}_+\}$ from having fully periodic accumulation points, which proves the claim.\medbreak   

\begin{claim}\label{nonexforbothorient_case1}
We have $-\pmb{\ell}',\pmb{\ell}' \in \nexpd(\eta)$.
\end{claim}

Indeed, let $z'_{per}$ be an accumulation point of $\{T^{t\vec{v}_{\pmb{\ell}'}}(T^u\eta) : t \in \bb{Z}_+\}$. As $(T^u\eta)\sob{\mathcal{K}}$ is periodic with period parallel to $\pmb{\ell}'$, then \(z'_{per}\sob{\mathcal{H}(\pmb{\ell}'_{\mathcal{K}})}\) is periodic with period parallel to $\pmb{\ell}'$. Thus, Pro\-po\-si\-tion~\ref{prop_dec_period_half_plane} implies that $z'_{per}$ is periodic with period parallel to $\pmb{\ell}'$, but, according to Claim~\ref{notfullyperiodic_case1}, not fully periodic. Therefore, from Boyle-Lind Theorem and Proposition~\ref{pps_par_antipar_exp} result that $-\pmb{\ell}',\pmb{\ell}' \in \nexpd(z'_{per}) \subset \nexpd(\eta)$, which proves the claim.\medbreak

\begin{claim}\label{Iminimal}
We have $I = \iota-1$.
\end{claim}
Indeed, suppose, by contradiction, that $I < \iota-1$. Since $(T^u\eta)\sob{\mathcal{K}}$ is fully periodic and \(\vec{v}_{\pmb{\ell}_{\iota-1}}\) points inward to the $(-\pmb{\ell},\pmb{\ell}_I)$-region \(\mathcal{K}\), the set $\{T^{t\vec{v}_{\pmb{\ell}_{\iota-1}}}(T^u\eta) : t \in \bb{Z}_+\}$ has a fully periodic accumulation point, which due to Lemma~\ref{lem_no_fullyperiodicaccpoint} is a contradiction and establishes the claim.\medbreak

Finally, from Claims \ref{nonexforbothorient_case1} and \ref{Iminimal} we get that \(-\pmb{\ell}_{\iota-1},\pmb{\ell}_{\iota-1} \in \nexpd(\eta),\) but as $\pmb{\ell}_{\iota-1}$ is antiparallel to $\pmb{\ell}_{\iota+m-1}$, we have \[-\pmb{\ell}_{\iota+m-1},\pmb{\ell}_{\iota+m-1} \in \nexpd(\eta),\] which establishes Lemma \ref{mainlemmamainthm} for Case 1.

\medbreak
\subsection*{Case 2} Suppose that, for any $E(\mathcal{S}_{\varphi})$-enveloped set $B \subset \bb{Z}^2$ and all $u \in \bb{Z}^2$ such that $(T^u\eta)\sob{B} = x_{per}\sob{B}$, one has $(T^u\eta)\sob{H_B(\pmb{\ell})} \neq x_{per}\sob{H_B(\pmb{\ell})}$.
\medbreak

Initially, note that $(x_{per})_g \neq (y_{per})_g$ for some $g \in \pmb{\ell}^{(-)} \cap \bb{Z}^2$ prevents both $x_{per}\sob{\mathcal{H}(\pmb{\ell}^{(-)})}$ and $y_{per}\sob{\mathcal{H}(\pmb{\ell}^{(-)})}$ from being fully periodic. Therefore, we may assume that $x_{per}\sob{\mathcal{H}(\pmb{\ell}^{(-)})}$ is not fully periodic. Due to item (i) of Lemma \ref{tech_lemma}, there exist an \((\pmb{\ell},\pmb{\ell}_J)\)-region \(\hat{A}_{\infty}\), with \(\iota+1 \leq J \leq \iota+m-1\), a configuration \(\vartheta \in X_{\eta}\) and \(\epsilon \in \bb{Z}_+\) such that
\begin{equation}\label{eq_existence_tau_case2_demonst}
	\vartheta\sob{\hat{A}_{\infty}^{(\epsilon)}} = \hat{x}_{per}\sob{\hat{A}_{\infty}^{(\epsilon)}}, \ \ \textrm{but} \ \ \vartheta\sob{\hat{A}_{\infty}^{(\epsilon+1)}} \neq \hat{x}_{per}\sob{\hat{A}_{\infty}^{(\epsilon+1)}},
\end{equation}
where $\hat{x}_{per} := T^{k\vec{v}_{\ell}}x_{per}$ for some \(k \in \bb{Z}_+\). To simplify the notation, we will write \(\pmb{\ell}' = \pmb{\ell}_J\).

\begin{claim}\label{semiambiguous_case2}
\(\vartheta\) is an \((\pmb{\ell}',\mathcal{T},+)\)-semi-ambiguous configuration for any translation $\mathcal{T}$ of $\mathcal{S}$ where $\mathcal{T} \backslash \pmb{\ell}'_{\mathcal{T}} \subset \hat{A}_{\infty}^{(\epsilon)}$, but $\mathcal{T} \not\subset \hat{A}_{\infty}^{(\epsilon)}$.
\end{claim}

Indeed, let \(h \in \bb{Z}^2\) denote a period for \(\vartheta\sob{\hat{A}_{\infty}^{(\epsilon)}}\) parallel to \(-\pmb{\ell}\). Since, due to (\ref{eq_existence_tau_case2_demonst}), \(\vartheta\sob{\hat{A}_{\infty}^{(\epsilon+1)}}\) is not periodic of period \(h\), the proof is identical to that one of Claim~\ref{semiambiguous_case1}.\medbreak

Due to Claim~\ref{semiambiguous_case2} and Lemma~\ref{lem_subrll'egion}, there exists an $(\pmb{\ell},\pmb{\ell}')$-region $\mathcal{K} \subset \hat{A}_{\infty}^{(\epsilon)}$ such that $\vartheta\sob{\mathcal{K}}$ is fully periodic with a period parallel to \(\pmb{\ell}\) and another one parallel to $\pmb{\ell}'$.

\begin{claim}\label{notfullyperiodic_case2}
The set $\{T^{t\vec{v}_{\pmb{\ell}'}}\vartheta : t \in \bb{Z}_+\}$ does not have fully periodic accumulation points.
\end{claim}

Indeed, the proof is identical to that one of Claim \ref{claim_notfullyper} with \(\pmb{\ell}\) instead of \(\pmb{\ell}_m\).\medbreak

\begin{claim}\label{nonexforbothorient_case2}
We have $-\pmb{\ell}',\pmb{\ell}' \in \nexpd(\vartheta)$.
\end{claim}

Indeed, let \(z'_{per}\) be an accumulation point of $\{T^{t\vec{v}_{\pmb{\ell}'}}\vartheta : t \in \bb{Z}_+\}$. Since \(\vartheta\sob{\mathcal{K}}\) is periodic with period parallel to \(\pmb{\ell}'\), then \(z'_{per}\sob{\mathcal{H}(\pmb{\ell}'_{\mathcal{K}})}\) is periodic with period parallel to \(\pmb{\ell}'\). Thus, Proposition \ref{prop_dec_period_half_plane} implies that $z'_{per}$ is periodic with period parallel to \(\pmb{\ell}'\), but, according to Claim \ref{notfullyperiodic_case2}, not fully periodic. Therefore, from Boyle-Lind Theorem and Proposition~\ref{pps_par_antipar_exp} result that $-\pmb{\ell}',\pmb{\ell}' \in \nexpd(z'_{per}) \subset \nexpd(\vartheta)$, which proves the claim.\medbreak

\begin{claim}\label{nonperiodic_case2}
The configuration \(\vartheta\) is non-periodic.
\end{claim}

Indeed, as $\vartheta\sob{\hat{A}^{(\epsilon)}_{\infty}} = \hat{x}_{per}\sob{\hat{A}^{(\epsilon)}_{\infty}}$ is periodic with period parallel to $\pmb{\ell}$ and the support line of \(\hat{A}^{(\epsilon)}_{\infty}\) determined by \(\pmb{\ell}\) coincides with \(\pmb{\ell}^{(-)}\), the fact that $x_{per}\sob{\mathcal{H}(\pmb{\ell}^{(-)})}$ and so $\hat{x}_{per}\sob{\mathcal{H}(\pmb{\ell}^{(-)})}$ is not fully periodic prevents the set $\{T^{-t\vec{v}_{\pmb{\ell}}}\vartheta : t \in \bb{Z}_+\}$ from having a\linebreak fully periodic accumulation point. Let \(z_{per}\) be an accumulation point of $\{T^{-t\vec{v}_{\pmb{\ell}}}\vartheta : t \in \bb{Z}_+\}$. As \(\vartheta\sob{\hat{A}^{(\epsilon)}_{\infty}}\) is periodic with period parallel to \(\pmb{\ell}\), then \(z_{per}\sob{\mathcal{H}(\pmb{\ell}^{(-)})}\) is periodic with period parallel to \(\pmb{\ell}\). Thus, Proposition \ref{prop_dec_period_half_plane} implies that $z_{per}$ is periodic with period parallel to \(\pmb{\ell}\), but not fully periodic, which means that $\ell \in \nexpl(z_{per}) \subset \nexpl(\vartheta)$. Therefore, from Claim \ref{nonexforbothorient_case2} we conclude that $\ell,\ell' \in \nexpl(\vartheta)$, which proves the claim.\medbreak

\begin{claim}\label{Jmaximal}
We have $J = \iota+m-1$.
\end{claim}

Indeed, suppose, by contradiction, that $J < \iota+m-1$. Since \(\vartheta \in X_{\eta}\), then $(X^{h_1}-1) \cdots (X^{h_m}-1)$ annihilates \(\vartheta\). So let \(\vartheta = \vartheta_1+\cdots+\vartheta_m\) be a \(\bb{Z}\)-periodic decomposition, where \(h_i \in \bb{Z}^2\) is a period some \(\eta_i\), with \(1 \leq i \leq m\) (see Theorem~\ref{theorKS}). Since, by assumption \(\eta\) is non-periodic and by Claim~\ref{nonperiodic_case2} \(\vartheta\) is non-periodic and all non-periodic configurations in \(X_{\eta}\) have the same order, then \(\vartheta = \vartheta_1+\cdots+\vartheta_m\) is a \(\bb{Z}\)-minimal periodic decomposition. Thus, as $\vartheta\sob{\mathcal{K}}$ is fully periodic and \(\vec{v}_{\pmb{\ell}_{\iota+m-1}}\) points inward to the $(\pmb{\ell},\pmb{\ell}_J)$-region \(\mathcal{K}\), the set $\{T^{t\vec{v}_{\pmb{\ell}_{\iota+m-1}}}\vartheta : t \in \bb{Z}_+\}$ has a fully periodic accumulation point, which due to Lemma~\ref{lem_no_fullyperiodicaccpoint} is a contradiction and establishes the claim.\medbreak

Finally, from Claims \ref{nonexforbothorient_case2} and \ref{Jmaximal} we get that \[-\pmb{\ell}_{\iota+m-1},\pmb{\ell}_{\iota+m-1} \in \nexpd(\vartheta) \subset \nexpd(\eta),\] which establishes Lemma \ref{mainlemmamainthm} for Case 2.

The case $|\mathcal{S} \cap -\pmb{\ell}_{\mathcal{S}}| < |\mathcal{S} \cap \pmb{\ell}_{\mathcal{S}}|$ follows from the previous one by considering the enumeration $\pmb{\hat{\ell}}_1,\ldots,\pmb{\hat{\ell}}_{2m}$ where \(\pmb{\hat{\ell}}_{\iota} = -\pmb{\ell}_{\iota}\). 
\end{proof}

Finally, by applying the previous lemma successively we conclude the proof of Theorem \ref{main_thm}. Indeed, by assumption \(\eta\) is non-periodic and so not fully periodic. Then, according to Theorem \ref{secondary_thm}, there exists a line \(\ell \in \nexpl(\eta)\) such that \(-\pmb{\ell},\pmb{\ell} \in \nexpd(\eta)\). Due to Lemma~\ref{lem_periods_annihilator}, we know that \(\ell\) contains some vector \(h_i\), which means that \(\pmb{\ell} = \pmb{\ell}_{\iota}\) for some \(1 \leq \iota \leq 2m\). 

By Lemma \ref{mainlemmamainthm} applied successively, we have that \[\pmb{\ell}_{\iota+m-1}, \pmb{\ell}_{\iota+2m-2}, \pmb{\ell}_{\iota+3m-3}, \ldots, \pmb{\ell}_{\iota+(m-1)m-m+1} \in \nexpd(\eta)\] and \[-\pmb{\ell}_{\iota+m-1}, -\pmb{\ell}_{\iota+2m-2}, -\pmb{\ell}_{\iota+3m-3}, \ldots, -\pmb{\ell}_{\iota+(m-1)m-m+1} \in \nexpd(\eta),\] where indices are taken modulo $2m$. However, since $\pmb{\ell}_{\iota+jm-j} =  \pmb{\ell}_{\iota+m-j}$ if $j$ is odd and $-\pmb{\ell}_{\iota+jm-j} = \pmb{\ell}_{\iota+m-j}$ if $j$ is even, then \[\pmb{\ell}_{\iota+1}, \pmb{\ell}_{\iota+2}, \ldots, \pmb{\ell}_{\iota+m-2}, \pmb{\ell}_{\iota+m-1} \in \nexpd(\eta).\] Similarly, since $-\pmb{\ell}_{\iota+jm-j} = -\pmb{\ell}_{\iota+m-j}$ if $j$ is odd and $\pmb{\ell}_{\iota+jm-j} = -\pmb{\ell}_{\iota+m-j}$ if $j$ is even, then \[-\pmb{\ell}_{\iota+1}, -\pmb{\ell}_{\iota+2}, \ldots, -\pmb{\ell}_{\iota+m-2}, -\pmb{\ell}_{\iota+m-1} \in \nexpd(\eta),\] which establishes Theorem~\ref{main_thm}.
\hfill $\Box$\medbreak

\vspace{5pt}

The next result is an immediate corollary from Theorems~\ref{main_thm} and \ref{main_theor_szabados}. 

\begin{corollary}\label{SzabadosConjem=3}
Let $\eta \in \mathcal{A}^{\bb{Z}^2}$, with $\mathcal{A} \subset \bb{Z}$, be a low convex complexity configuration. If $\eta = \eta_1+\eta_2+\eta_3$ is a $\bb{Z}$-minimal periodic decomposition, then
\begin{enumerate}[(i)]\setlength{\itemsep}{6pt}
	\item Szabados's conjecture holds;
	\item $\ell \in \nexpl(\eta)$ if and only if the antiparallel oriented lines $-\pmb{\ell},\pmb{\ell} \in \nexpd(\eta)$.
\end{enumerate} 
\end{corollary}

\end{document}